\documentclass[twoside,11pt]{article}

\usepackage{graphicx, subfigure}
\usepackage[top=1in,bottom=1in,left=1in,right=1in]{geometry}
\usepackage[sort&compress]{natbib} 
\usepackage{amsfonts, amsmath, amssymb, amsthm, bbm}
\usepackage[normalem]{ulem}
\usepackage{comment}
\usepackage{eufrak}

\usepackage{color}
\definecolor{darkred}{RGB}{100,0,0}
\definecolor{darkgreen}{RGB}{0,100,0}
\definecolor{darkblue}{RGB}{0,0,150}

\usepackage{hyperref}
\hypersetup{colorlinks=true, linkcolor=darkred, citecolor=darkgreen, urlcolor=darkblue}
\usepackage{url}

\newtheorem{thm}{Theorem}
\newtheorem{prp}{Proposition}
\newtheorem{lem}{Lemma}
\newtheorem{cor}{Corollary}
\theoremstyle{remark}
\newtheorem{rem}{Remark}

\def\beq{\begin{equation}}
\def\eeq{\end{equation}}
\def\beqn{\begin{eqnarray*}}
\def\eeqn{\end{eqnarray*}}
\def\bitem{\begin{itemize}}
\def\eitem{\end{itemize}}
\def\benum{\begin{enumerate}}
\def\eenum{\end{enumerate}}
\def\bmult{\begin{multline*}}
\def\emult{\end{multline*}}
\def\bcenter{\begin{center}}
\def\ecenter{\end{center}}

\newcommand{\thmref}[1]{Theorem~\ref{thm:#1}}
\newcommand{\prpref}[1]{Proposition~\ref{prp:#1}}
\newcommand{\corref}[1]{Corollary~\ref{cor:#1}}
\newcommand{\lemref}[1]{Lemma~\ref{lem:#1}}
\newcommand{\secref}[1]{Section~\ref{sec:#1}}
\newcommand{\figref}[1]{Figure~\ref{fig:#1}}

\DeclareMathOperator{\tr}{tr}

\def\cB{\mathcal{B}}
\def\cC{\mathcal{C}}

\def\cH{\mathcal{H}}

\def\cN{\mathcal{N}}

\def\cS{\mathcal{S}}

\def\cV{\mathcal{V}}

\def\bA{\mathbf{A}}
\def\bB{\mathbf{B}}

\def\bF{\mathbf{F}}

\def\bI{\mathbf{I}}

\def\bR{\mathbf{R}}

\def\bW{\mathbf{W}}

\newcommand\bSigma{{\boldsymbol\Sigma}}

\newcommand\bGamma{{\boldsymbol\Gamma}}

\def\bbE{\mathbb{E}}

\def\bbI{\mathbb{I}}

\def\bbN{\mathbb{N}}

\def\bbP{\mathbb{P}}

\def\bbR{\mathbb{R}}

\def\bbZ{\mathbb{Z}}

\newcommand{\E}{\operatorname{\mathbb{E}}}
\renewcommand{\P}{\operatorname{\mathbb{P}}}
\newcommand{\Var}{\operatorname{Var}}
\newcommand{\Cov}{\operatorname{Cov}}
\newcommand{\Cor}{\operatorname{Cor}}

\newcommand{\pr}[1]{\mathbb{P}\left\{#1\right\}}
\newcommand{\var}[1]{\operatorname{Var}\left(#1\right)}

\newcommand\indep{\protect\mathpalette{\protect\independenT}{\perp}}
\def\independenT#1#2{\mathrel{\rlap{$#1#2$}\mkern2mu{#1#2}}}

\def\iid{\stackrel{\rm iid}{\sim}}

\def\eps{\epsilon}

\def\defeq{\stackrel{\rm def}{=}}

\newcommand{\PROB}{\bbP}
\newcommand{\EXP}{\bbE}

\def\tr{\operatorname{Tr}}

\def\ar{{\rm AR}}
\newcommand{\1}{{\rm 1}\kern-0.24em{\rm I}}

\newcommand{\IND}[1]{\bbI\{ #1 \}}
\def\X{X}
\def\Y{Y}

\begin{document}

\title{Detecting 
Markov Random Fields Hidden in White Noise}
\author{
Ery Arias-Castro\footnote{Department of Mathematics, University of California, San Diego} \and 
S\'ebastien Bubeck\footnote{Department of Operations Research and Financial Engineering, Princeton University} \and
G\'abor Lugosi\footnote{ICREA and Department of Economics, Universitat Pompeu Fabra} \and
Nicolas Verzelen\footnote{(corresponding author) INRA, UMR 729 MISTEA, F-34060 Montpellier, FRANCE}
} 
\date{}
\maketitle

\begin{abstract}
Motivated by change point problems in time series and the detection of textured objects in images, we consider the problem of detecting a piece of a Gaussian Markov random field hidden in white Gaussian noise.  We derive minimax lower bounds and propose near-optimal tests.  
\end{abstract}

\section{Introduction} \label{sec:intro}

Anomaly detection is important in a number of applications, including surveillance and environment monitoring systems using sensor networks, object tracking from video or satellite images, and tumor detection in medical imaging.  The most common model is that of an object or signal of unusually high amplitude hidden in noise.  
In other words, one is interested in detecting the presence of an object in which the mean of the signal is different from that of the background.
We refer to this as the \emph{detection-of-means} problem.
In many situations, anomaly manifests as unusual dependencies in the data. 
This  \emph{detection-of-correlations} problem is the one that we consider in this paper.

\subsection{Setting and hypothesis testing problem} \label{sec:setting}
It is common to model dependencies by a Gaussian random field $\X = (X_i : i \in \cV)$, where $\cV \subset \cV_\infty$ is of size $|\cV| = n$, while $\cV_\infty$ is countably infinite.  
We focus on the important example of a $d$-dimensional integer lattice
\beq \label{lattice}
\cV = \{1, \dots, m\}^d \subset \cV_\infty = \bbZ^d.
\eeq
%See \citep{correlation-paths} for a 
%Furthermore, the $X_i$'s are assumed to be jointly normal with zero mean and unit variance.  We emphasize that anomalies come in the form of correlations between some of these variables.  

We formalize the task of detection as the following hypothesis testing problem.
One observes a realization of $\X = (X_i : i \in \cV)$, where the $X_i$'s are known to be standard normal.
Under the null hypothesis $\cH_0$, the $X_i$'s are independent.
Under the alternative hypothesis $\cH_1$, the $X_i$'s are correlated in one of the following ways. 
Let $\cC$ be a class of subsets of $\cV$. 
Each set $S\in \cC$ represents a possible anomalous subset of the components of $\X$.
% (A prototypical example of a class $\cC$ of subsets in the lattice is that of all hypercubes of $\cV$ of size $k = \ell^d$.)
Specifically, when $S \in \cC$ is the anomalous subset of nodes, each $X_i$ with $i \notin S$ is still independent of all the other variables, while $(X_i : i \in S)$ coincides with $(Y_i : i \in S)$, where $\Y=(Y_i : i \in \cV_\infty)$ is
a stationary Gaussian Markov random field.  
We emphasize that, in this formulation, the anomalous subset $S$ is only known to belong to $\cC$.

We are thus addressing the problem of detecting a region of a Gaussian Markov random field against a background of white noise.  
This testing problem models important detection problems such as the detection of a piece of a time series in a signal and the detection of a textured object in an image, which we describe below. 
Before doing that, we further detail the model and set some foundational notation and terminology.

\subsection{Tests and minimax risk} \label{sec:tests}
We denote the distribution of $\X$ under $\cH_0$ by $\PROB_0$.
The distribution of the zero-mean stationary Gaussian Markov random
field $\Y$ is determined by its covariance operator
$\bGamma=(\bGamma_{i,j} : i,j\in \cV_\infty)$ defined by
$\bGamma_{i,j}=\EXP[Y_i Y_j]$.
We denote the distribution of $\X$ under $\cH_1$ by $\PROB_{S,\bGamma}$ when $S\in \cC$ is the anomalous set and $\bGamma$ is
the covariance operator of the Gaussian Markov random field $\Y$.

A \emph{test} is a measurable function $f: \bbR^\cV \to \{0,1\}$. When $f(\X)=0$, 
the test accepts the null hypothesis and it rejects it otherwise.
The probability of \emph{type I} error of a test $f$ is $\PROB_0\{f(\X)=1\}$.
When $S\in \cC$ is the anomalous set and $\Y$ has covariance operator $\bGamma$, the probability of \emph{type II} error is $\PROB_{S,\bGamma}\{f(\X)=0\}$.
In this paper we evaluate tests based on their \emph{worst-case risks}.
The risk of a test $f$ corresponding to a covariance operator $\bGamma$ and class of sets $\cC$ is defined as 
\beq \label{risk-known-f}
R_{\cC,\bGamma}(f) = \PROB_0\{f(\X)=1\} +  \max_{S \in \cC} \, \PROB_{S,{\bGamma}}\{f(\X)=0\}~.
\eeq
Defining the risk this way is meaningful when the distribution of $\Y$ is known, 
meaning that $\bGamma$ is available to the statistician. 
In this case, the minimax risk is defined as 
\beq \label{risk-known}
R^*_{\cC,\bGamma} = \inf_f R_{\cC,\bGamma}(f)~,
\eeq
where the infimum is over all tests $f$.
When $\bGamma$ is only known to belong to some class of covariance operators $\mathfrak{G}$, it is more meaningful to define the risk of a test $f$ as
\beq \label{risk-unknown-f}
R_{\cC,\mathfrak{G}}(f) = \PROB_0\{f(\X)=1\} +  \max_{\bGamma \in \mathfrak{G}} \max_{S \in \cC} \, \PROB_{S,{\bGamma}}\{f(\X)=0\}~.
\eeq
The corresponding minimax risk is defined as
\beq \label{risk-unknown}
R_{\cC,\mathfrak{G}}^* = \inf_f R_{\cC,\mathfrak{G}}(f)~.
\eeq
In this paper we consider situations in which the covariance operator $\Gamma$ is known
(i.e., the test $f$ is allowed to be constructed using this information) and other situations
when $\Gamma$ is unknown but it is assumed to belong to a class $\mathfrak{G}$.
When $\bGamma$ is known (resp.~unknown), we say that a test $f$ \emph{asymptotically separates the two hypotheses} if $R_{\cC,\bGamma}(f) \to 0$ (resp.~$R_{\cC,\mathfrak{G}}(f) \to 0$), and we say that the hypotheses \emph{merge asymptotically} if $R_{\cC,\bGamma}^* \to 1$ (resp.~$R_{\cC,\mathfrak{G}}^* \to 1$), as $n = |\cV| \to \infty$.  We note that, as long as $\bGamma \in \mathfrak{G}$, $R_{\cC,\bGamma}^* \le R_{\cC,\mathfrak{G}}^*$, and that $R_{\cC,\mathfrak{G}}^* \le 1$, since the test $f \equiv 1$ (which always rejects) has risk equal to $1$.

At a high-level, our results are as follows.  We characterize the minimax testing risk for both known  ($R^*_{\cC,\bGamma}$) and unknown ($R_{\cC,\mathfrak{G}}^*$) covariances when the anomaly is a Gaussian Markov random field. More precisely, we give conditions on $\bGamma$ or $\mathfrak{G}$ enforcing the hypotheses to merge asymptotically so that detection problem is nearly impossible. Under nearly matching conditions, we exhibit tests that asymptotically separate the hypotheses. Our general results are illustrated in the following subsections.

\subsection{Example: detecting a piece of time series} \label{sec:intro-tseries}
As a first example of the general problem described above, consider the case of observing a time series $X_1,\ldots,X_n$. 
This corresponds to the setting of the lattice \eqref{lattice} in dimension $d=1$.
Under the null hypothesis, the $X_i$'s are i.i.d.\ standard normal random variables.
We assume that the anomaly comes in the form of temporal correlations over an (unknown) interval $S = \{i+1, \dots, i+k\}$ of, say, known length $k<n$.  Here, $i\in \{0,1\ldots,n-k\}$ is thus unknown.  Specifically, when $S$ is the anomalous interval, $(X_{i+1}, \dots, X_{i+k}) \sim (Y_{i+1}, \dots, Y_{i+k})$, where
$(Y_i: i \in \bbZ)$ is an autoregressive process of order $h$ (abbreviated $\ar_h$) with zero mean and unit variance, that is,
\beq \label{ARp}
Y_i = \psi_1 Y_{i-1} + \cdots + \psi_h Y_{i-h} + \sigma Z_i, \quad \forall i \in \bbZ,
\eeq
where $(Z_i: i \in \bbZ)$ are i.i.d.\ standard normal random variables, $\psi_1, \dots, \psi_h \in \bbR$ are the coefficients of the process---assumed to be stationary---and $\sigma>0$ is such that $\Var(Y_i) = 1$ for all $i$.  Note that $\sigma$ is a function of $\psi_1, \dots, \psi_h$, so that the model has effectively $h$ parameters. 
It is well-known that the parameters $\psi_1,\ldots,\psi_h$ define a stationary process when the roots of the polynomial $z^p - \sum_{i=1}^p \psi_i z^{p-i}$ in the complex plane lie within the open unit circle. See \cite{MR1093459} for a standard reference on time series.

In the simplest setting $h=1$ and the parameter space for $\psi$ is $(-1,1)$. Then, the hypothesis testing problem is to distinguish
\[
\cH_0: X_1, \dots, X_n \iid \cN(0,1),
\]
versus
\[\cH_1: \exists i \in \{0,1,\ldots,n-k\} \text{ such that }\]
\[X_1, \dots, X_i, X_{i+k+1}, \dots, X_n \iid \cN(0,1) \]
and
$(X_{i+1}, \dots, X_{i+k})$ is independent of $X_1, \dots, X_i, X_{i+k+1}, \dots, X_n$
with
\[
%X_{i+1} \sim \cN(0,1), \quad 
X_{i+j+1} - \psi X_{i+j} \iid \cN(0,1-\psi^2), \quad \forall j \in \{1, \dots, k-1\}\ .
\]
Typical realizations of the observed vector under the null and alternative hypotheses
are illustrated in \figref{timeseries}.  

\begin{figure}[htbp]
\centering
\includegraphics[width=0.9\linewidth]{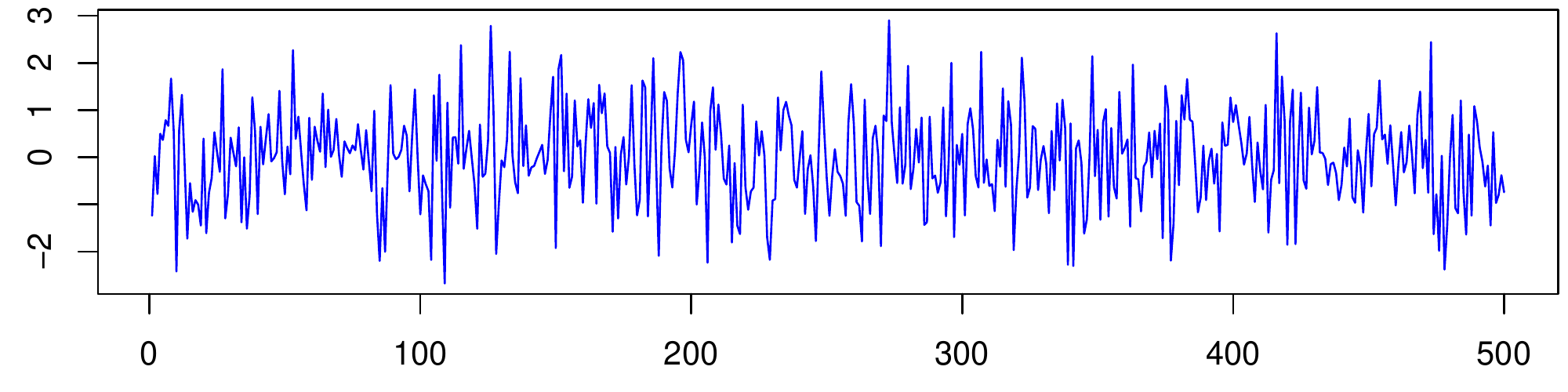} \\ 
\includegraphics[width=0.9\linewidth]{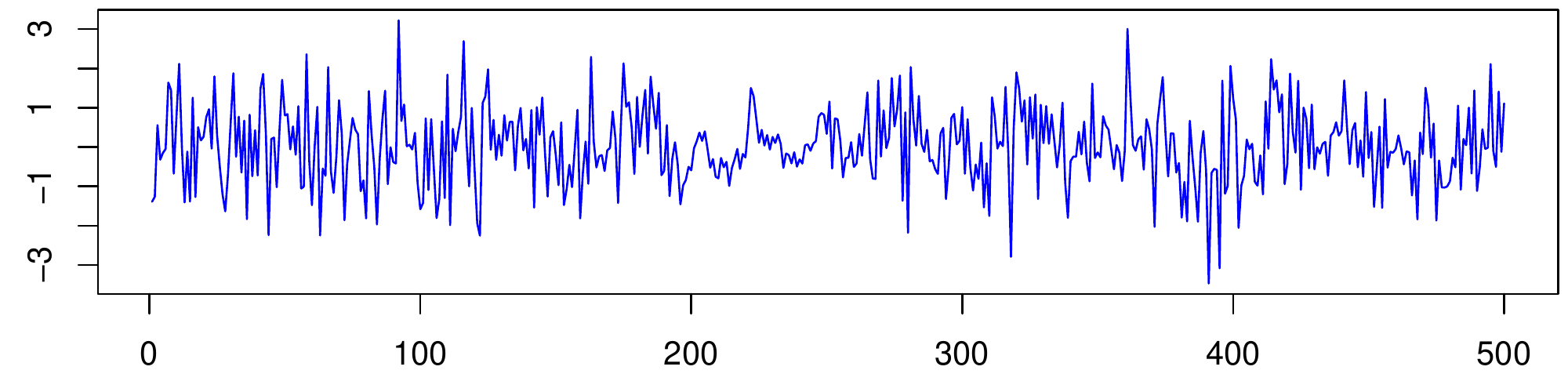}
\caption{Top: a realization of the observed time series under the null hypothesis (white noise).  Bottom: a realization under the alternative with anomalous interval $S = \{201, \dots, 250\}$, assuming an $\ar_1$ covariance model with parameter $\psi = 0.9$.}
\label{fig:timeseries}
\end{figure}

Gaussian autoregressive processes and other correlation models
are special cases of Gaussian Markov random fields, and therefore this setting is a special case of our general framework, with $\cC$ being the class of discrete intervals of length $k$. 
In the simplest case, the length of the anomalous interval is known beforehand.  In more complex settings, it is unknown, in which case $\cC$ may be taken to be the class of all intervals within $\cV$ of length at least $k_{\rm min}$.

This testing problem has been extensively studied in the slightly different context of change-point analysis, where under the null hypothesis $X_1, \dots, X_n$ are generated from an $\ar_h(\psi^0)$ process for some $\psi^0\in \mathbb{R}^h$, while under the alternative hypothesis there is an $i \in \cV$ such that $X_1, \dots, X_i$ and $X_{i+1}, \dots, X_n$ are generated from $\ar_h(\psi^0)$ and $\ar_h(\psi^1)$, with $\psi^0 \ne \psi^1$, respectively.  The order $h$ is often given.  In fact, instead of assuming autoregressive models, nonparametric models are often favored.  See, for example, \cite{MR0269062,MR1331669,springerlink:10.1007/BF00970969,MR1791905,MR2597589,MR2301477,MR809433,MR1200409} and many other references therein.  These papers often suggest maximum likelihood tests whose limiting distributions are studied under the null and (sometimes fixed) alternative hypotheses.  For example, in the special case of $h=1$, such a test would reject $\cH_0$  when $|\hat \psi|$ is large, where $\hat \psi$ is the maximum likelihood estimate for $\psi$.  
In particular, from \cite{MR809433}, we can speculate that such a test can asymptotically separate the hypotheses in the simplest setting described above when $\psi k^\alpha \to \infty$ for some $\alpha < 1/2$ fixed.  
See also \cite{MR2301477,MR2597589} for power analyses against fixed alternatives.

Our general results imply the following in the special case when the anomaly comes in the form of an autoregressive process with unknown parameter $\psi \in \bbR^h$.  We note that the order of the autoregressive model $h$ is allowed to grow with $n$ in this asymptotic result.  

\begin{cor}\label{cor:AR}
Assume $n, k \to \infty$, and that $h = o\big(\sqrt{k/\log(n)}\wedge k^{1/4}\big)$.
Denote by $\mathfrak{F}(h,r)$ the class of covariance operators corresponding to $AR_h$ processes with valid parameter $\psi=(\psi_1,\ldots, \psi_h)$ satisfying $\|\psi\|_2^2\geq r^2$.
Then $R_{\cC,\mathfrak{F}(h,r)}^* \to 1$ when
\beq\label{AR1}
r^2 \leq C_1 \big(\log(n/k)/k + \sqrt{h\log(n/k)}/k\big)\ .
\eeq
Conversely, if $f$ denotes the pseudo-likelihood test of Section~\ref{sec:fisher}, then $R_{\cC,\mathfrak{F}(h,r)}(f) \to 0$ when 
\beq\label{AR2}
r^2 \geq C_2 \big(\log(n)/k + \sqrt{h\log(n)}/k\big)\ .
\eeq
In both cases, $C_1$ and $C_2$ denote numerical constants.
\end{cor}

\begin{rem}
In the interesting setting where $k = n^\kappa$ for some $\kappa > 0$ fixed, the lower and upper bounds provided by \corref{AR} match up to a multiplicative constant that depends only on $\kappa$. 
\end{rem}

Despite an extensive literature on the topic, we are not aware of any other minimax optimality result for time series detection.

\subsection{Example: detecting a textured region} \label{sec:intro-images}
%\medskip \noindent 
%{\bf Detection of textured objects.}
In image processing, the detection of textured objects against a textured background is relevant in a number of applications, such as in the detection of local fabric defects in the textile industry by automated visual inspection~\citep{4418522}, the detection of a moving object in a textured background~\citep{yilmaz2006object,Kim2005172}, the identification of tumors in medical imaging~\citep{1229852,breast-tumor}, the detection of man-made objects in natural scenery~\citep{10.1109/CVPR.2003.1211345}, the detection of sites of interest in archeology \citep{litton} and of weeds in crops \citep{MR1959083}.  In all these applications, the object is generally small compared to the size of the image.  

Common models for texture include Markov random fields~\citep{4767341}
and joint distributions over filter banks such as wavelet
pyramids~\citep{531803,springerlink:10.1023/A:1026553619983}.  We
focus here on textures that are generated via Gaussian Markov random
fields \citep{1164641,springerlink:10.1023/A:1007925832420}.  Our goal
is to detect a textured object hidden in white noise.  
For this discussion, we place ourselves in the lattice setting \eqref{lattice} in dimension $d=2$.
% so that dimension $d=2$ and assume that
%$\cV=\{1,\ldots,\sqrt{n}\}^2$ is a square in the integer lattice $\cV_{\infty}=\bbZ^2$.  This corresponds to .  
Just like before, under $\cH_0$, the $(X_i : i\in \cV)$ are independent standard
normal random variables. Under $\cH_1$, 
when the region $S \subset \cV$ is anomalous, the $(X_i : i\notin S)$ are still i.i.d.\ standard normal, while $(X_i : i \in S) \sim (Y_i : i \in S)$, where $(Y_i: i \in \bbZ^2)$ is such that
for each $i\in \bbZ^2$, the conditional distribution of $Y_i$ given the rest of the variables $Y^{(-i)} := (Y_j : j \ne i)$ is normal with mean
\beq \label{markov-lattice}
\sum_{(t_1,t_2)\in [-h,h]^2\setminus \{(0,0)\}} \phi_{t_1, t_2} Y_{i + (t_1, t_2)}
\eeq
and variance $\sigma_{\phi}^2$, where the $\phi_{t_1,t_2}$'s are the coefficients of the process and $\sigma_{\phi}$ is such that $\Var(Y_i) = 1$ for all $i$. 
The set of valid parameters $\phi$ is defined in Section~\ref{sec:prelim_gmrf}. A simple sufficient condition is $\|\phi\|_1= \sum_{(t_1,t_2)\in [-h,h]^2\setminus \{(0,0)\}} |\phi_{t_1, t_2}|<1$. 
In this model, the \emph{dependency neighborhood} of $i \in \bbZ^2$ is $i + [-h,h]^2 \cap \bbZ^2$.
One of the simplest cases is when $h=1$ and $\phi_{t_1, t_2} = \phi$ when $(t_1,t_2) \in \{(\pm1,0), (0,\pm1)\}$ for some $\phi \in (-1/4, 1/4)$, and the anomalous region is a discrete square; see Figure \ref{fig:texture} for a realization of the resulting process.

This is a special case of our setting.
While intervals are natural in the case of time series, squares are rather restrictive models of anomalous regions in images.  We consider instead the ``blob-like'' regions (to be defined later) that include convex and star-shaped regions.  

\begin{figure}[htbp]
\centering
\includegraphics[width=.30\linewidth]{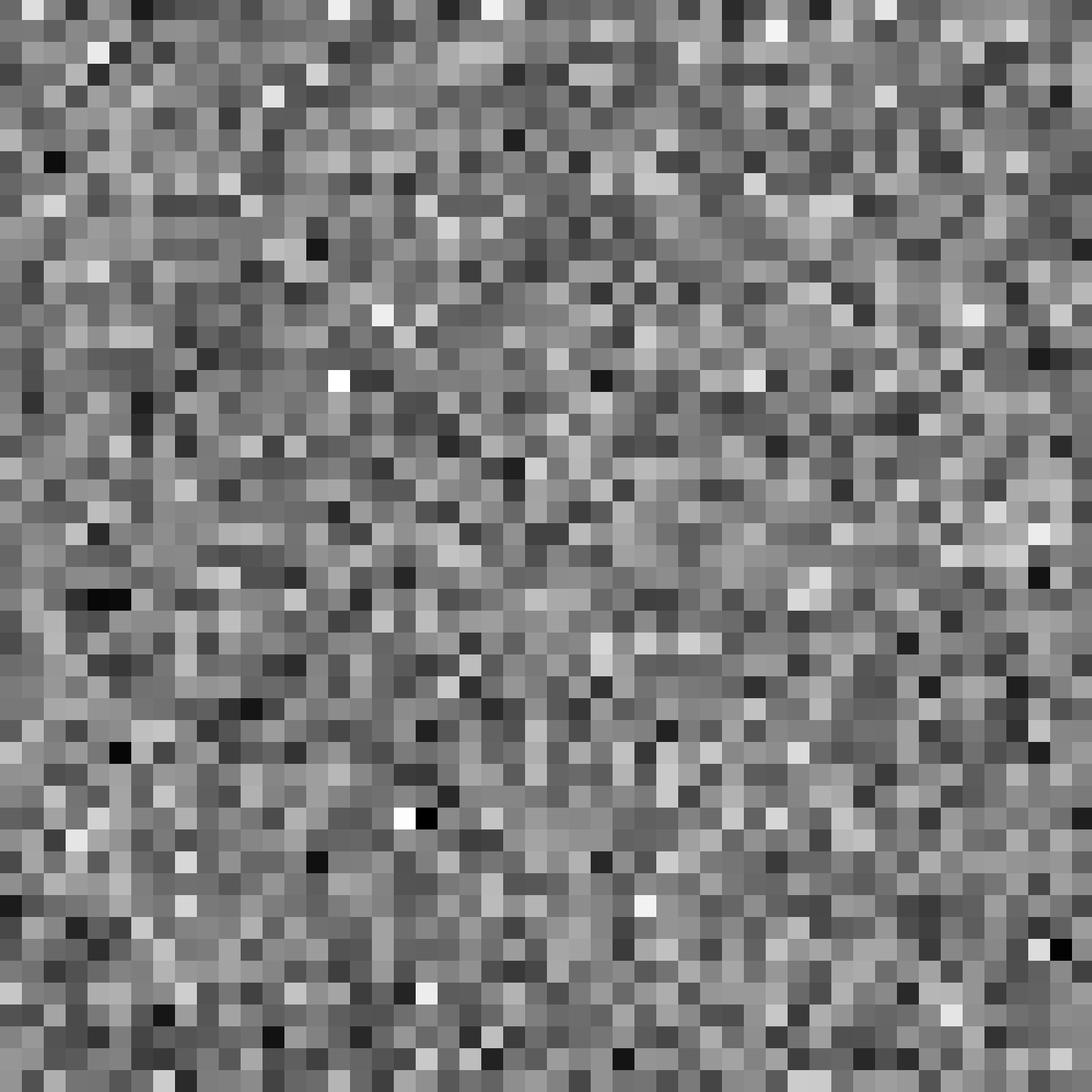} \ 
\includegraphics[width=.30\linewidth]{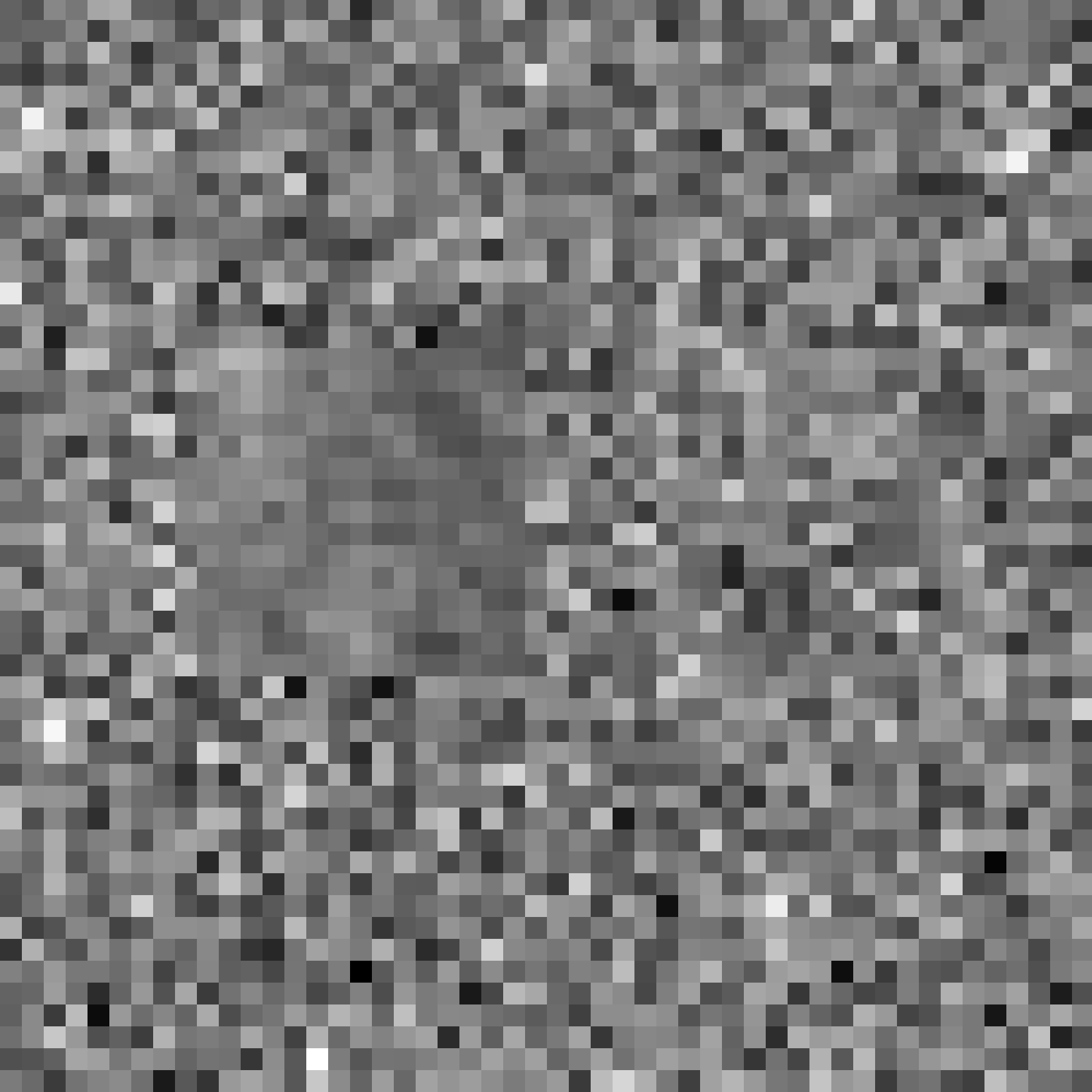} \ 
\caption{Left: white noise, no anomalous region is present.  Right: a squared anomalous region is present.  In this example on the $50 \times 50$ grid, the anomalous region is a $15 \times 15$ square piece from a Gaussian Markov random field with neighborhood radius $h=1$ and coefficient vector $\phi_{t_1, t_2} = \phi := \frac14(1-10^{-4})$ when $(t_1,t_2) \in \{(\pm1,0), (0,\pm1)\}$, and zero otherwise.}   
\label{fig:texture}
\end{figure}

A number of publications address the related problems of texture classification \citep{kervrann1995markov,springerlink:10.1023/A:1007925832420,varma05} and texture segmentation \citep{jain1991unsupervised,buhmann98,MR1966045,sharon,malik-belongie-al}.  In fact, this literature is quite extensive.  Only very few papers address the corresponding change-point problem \citep{shahrokni,palenichka:158} and we do not know of any theoretical results in this literature. 
Our general results (in particular, Corollary \ref{cor:lower_hypercube}) imply the following.

\begin{cor}\label{cor:image}
Assume $n, k \to \infty$, and that $h = o\big(\sqrt{k/\log(n)}\wedge k^{1/5}\big)$. 
Denote by $\mathfrak{G}(h,r)$ the class of covariance operators corresponding to stationary Gaussian Markov Random Fields with valid parameter (see Section~\ref{sec:prelim_gmrf} for more details) $\phi=(\phi_{i,j})_{(i,j)\in \{-h,\ldots, h\}^2\setminus \{0\}}$ satisfying $\|\phi\|_2^2\geq r^2$.
Then $R_{\cC,\mathfrak{G}(h,r)}^* \to 1$ when
\beq\label{CAR1}
r^2 \leq C_1\left[ \frac{\log(n/k)}{k} + \frac{\sqrt{h^2\log(n/k)}}{k}\right]\ .
\eeq
Conversely, if $f$ denotes the pseudo-likelihood test of Section~\ref{sec:fisher}, then $R_{\cC,\mathfrak{G}(h,r)}(f) \to 0$ when 
\beq\label{CAR2}
r^2 \geq C_2 \left[ \frac{\log(n/k)}{k} + \frac{\sqrt{h^2\log(n/k)}}{k}\right]\ .
\eeq
In both cases, $C_1$ and $C_2$ denote positive numerical constants.
\end{cor}

Informally, the lower bound on the magnitude of the coefficient vector $\phi$, namely $r^2$, quantifies the extent to which the variables $Y_i$ are explained by the rest of variables $Y^{(-i)}$ as in \eqref{markov-lattice}. 

Although not in the literature on change-point or object detection,  \cite{anandkumar2009detection} is the only other paper developing theory in a similar context.  It considers a spatial model where points $\{x_i, i \in [N]\}$ are sampled uniformly at random in some bounded region and a nearest-neighbor graph is formed.
On the resulting graph, variables are observed at the nodes.
Under the (simple) null hypothesis, the variables are i.i.d.~zero mean normal.  Under the (simple) alternative, the variables arise from a Gaussian Markov random with covariance operator of the form $\Gamma_{i,j} \propto g(\|x_i - x_j\|)$, where $g$ is a known function.  The paper analyzes the large-sample behavior of the likelihood ratio test.

\subsection{More related work}

As we mentioned earlier, the detection-of-means setting is much more prevalent in the literature.  When the anomaly has no a priori structure, the problem is that of multiple testing; see, for example, \cite{Ingster99,baraud,MR2065195} for papers testing the global null hypothesis.  Much closer to what interests us here, the problem of detecting objects with various geometries or combinatorial properties has been extensively analyzed, for example, in some of our earlier work \citep{maze,combin,cluster} and elsewhere \citep{MR2604703,morel}.  We only cite a few publications that focus on theory.  The applied literature is vast; see \cite{cluster} for some pointers.

Despite its importance in practice, as illustrated by the examples and references given in Sections~\ref{sec:intro-tseries} and~\ref{sec:intro-images}, the detection-of-correlations setting has received comparatively much less attention, at least from theoreticians.  Here we find some of our own work \citep{correlation-detect,multidim}.  
In the first of these papers, we consider a sequence $X_1, \dots, X_n$ of standard normal random variables.  Under the null, they are independent.  Under the alternative, there is a set $S$ in a class of interest $\cC$ where the variables are correlated.  We consider the unstructured case where $\cC$ is the class of all sets of size $k$ (given) and also various structured cases, and in particular, that of intervals.  This would appear to be the same as in the present lattice setting in dimension $d=1$, but the important difference is that that correlation operator $\bGamma$ is not constrained, and in particular no Markov random field structure is assumed.  
The second paper extends the setting to higher dimensions, thus testing whether some coordinates of a high-dimensional Gaussian vector are correlated or not.  
When the correlation structure in the anomaly is arbitrary, the setting overlaps with that of sparse principal component analysis \citep{berthet,cai2013sparse}.
The problem is also connected to covariance testing in high-dimensions; see, e.g., \cite{cai2013optimal}.
We refer the reader to the above-mentioned papers for further references.

\subsection{Contribution and content}

The present paper thus extends previous work on the detection-of-means setting to the detection-of-correlations setting in the (structured) context of detecting signals/objects in time series/images.  The paper also extends some of our own work on the detection-of-correlations to Markov random field models, which are typically much more appropriate in the context of detection in signals and images. 
The theory in the detection-of-correlations setting is more complicated than in the the detection-of-means setting, and in particular deriving exact minimax (first-order) results remains an open problem.  Compared to our previous work on the detection-of-correlations setting, the Markovian assumption makes the problem significantly more complex as it requires handling Markov random fields which are conceptually more complex objects.  As a result, the proof technique is by-and-large novel, at least in the detection literature.

The rest of the paper is organized as follows.
In \secref{prelim} we lay down some foundations on Gaussian Markov Random Fields, and in particular, their covariance operators, and we also derive a general minimax lower bound that is used several times in the paper.
In the remainder of the paper, we consider detecting correlations in a finite-dimensional lattice \eqref{lattice}, which includes the important special cases of time series and textures in images.
We establish lower bounds, both when the covariance matrix is known (Section~\ref{sec:known}) or unknown (Section~\ref{sec:unknown}) and  propose test procedures that are shown to achieve the lower bounds up to multiplicative constants.  In Section~\ref{sec:cubes}, we specialize our general results to specific classes of anomalous regions such as classes of cubes, and more generally, ``blobs.'' 
In \secref{discussion} we outline possible generalizations and further work.
The proofs are gathered in \secref{proofs}.

\section{Preliminaries}
\label{sec:prelim}

In this paper we derive upper and lower bounds for the minimax risk, both when $\bGamma$ is known as in \eqref{risk-known} and when it is unknown as in \eqref{risk-unknown}, the latter requiring a substantial amount of additional work.
For the sake of exposition, we sketch here the general strategy for obtaining minimax lower bounds by adapting the general strategy initiated in \cite{ingster93a} to detection-of-correlation problems. This allows us to separate the technique used to derive minimax lower bounds from the technique required to handle Gaussian Markov random fields.

\subsection{Some background on Gaussian Markov random fields}\label{sec:prelim_gmrf}

We elaborate on the setting described in Sections~\ref{sec:setting} and~\ref{sec:tests}. 
As the process $Y$ is indexed by $\mathbb{Z}^d$, note that all the indices $i$ of $\phi$ and $\bGamma$ are $d$-dimensional. 
Given a positive integer $h$, denote by $\bbN_h$ the integer lattice $\{-h, \dots, h\}^d \setminus\{0\}^d$ with $(2h+1)^d-1$ nodes. For any nonsingular covariance operator $\bGamma$ of a stationary Gaussian Markov random field over $\mathbb{Z}^d$ with unit variance and neighborhood $\bbN_h$, there exists a unique vector $\phi$ indexed by the nodes of $\bbN_h$ satisfying $\phi_i=\phi_{-i}$ such that, for all $i,j\in \bbZ^d$,  
\beq
\bGamma^{-1}_{i,j}/\bGamma^{-1}_{i,i}= 
\begin{cases}
- \phi_{i-j}&\text{ if } 1\leq |i-j|_{\infty}\leq h,\\
1 &\text{ if }i=j,\\
0 & \text{otherwise~,}                                                                          \end{cases}
\label{eq:conditions}
\eeq
where $\bGamma^{-1}$ denotes the inverse of the covariance operator
$\bGamma$. 
Consequently, there exists a bijective map from the collection of
invertible covariance operators of stationary Gaussian Markov random fields over $\mathbb{Z}^d$ with unit variance and neighborhood $\bbN_h$ to some subset $\Phi_h\subset \bbR^{\bbN_h}$. Given $\phi\in \Phi_h$, $\bGamma(\phi)$ denotes the unique covariance operator satisfying $\bGamma_{i,i}=1$ and \eqref{eq:conditions}. 
It is well known that $\Phi_h$ contains the set of vectors $\phi$ whose $\ell_1$-norm is smaller than one, that is,
\[
\{\phi \in \bbR^{\bbN_h} : \|\phi\|_1 < 1\} \subset \Phi_h\ ,
\]
as the corresponding operator $\bGamma^{-1}(\phi)$ is diagonally dominant in that case. 
In fact, the parameter space $\Phi_h$ is characterized by the Fast Fourier Transform (FFT) as follows
\[ \Phi_h=\Big\{\phi:\quad 1+ \sum_{1\leq |i|_{\infty}\leq h  } \phi_i \cos(\langle i,\omega\rangle) >0,\quad \forall \omega \in (-\pi,\pi]^{d}\Big\}\ ,\]
where  and $i\in \mathbb{Z}^d$ and $\langle \cdot , \cdot\rangle$ denotes the scalar product in $\mathbb{R}^d$. 
The interested reader  is referred to \cite[Sect.1.3]{MR1344683} or \cite[Sect.2.6]{MR2130347} for further details and discussions.
For $\phi\in \Phi_h$, define $\sigma^{2}_{\phi}=1/\bGamma _{i,i}^{-1}(\phi)$. 

The correlated process $Y=(Y_i : i\in \bbZ^d)$ is centered Gaussian with covariance operator $\bGamma(\phi)$ is such that, for each $i\in \mathbb{Z}^d$, the conditional distribution of $Y_i$
given the rest of the variables $Y^{(-i)}$ is
\beq\label{eq:conditional_definition}
Y_{i}|Y^{(-i)}  \ \sim \ \cN \Big(\sum_{j\in \bbN_h}
\phi_{j}Y_{i+j}, \sigma^2_{\phi}\Big) \ .
\eeq

Define the \emph{$h$-boundary} of $S$, denoted $\Delta_h(S)$, as the collection of vertices in $S$ whose distance to $\bbZ^d\setminus S$ is at most $h$. 
We also define the \emph{$h$-interior} $S$ as $S^h = S\setminus \Delta_h(S)$. 
If $S\subset \cV$ is a finite set, we denote by $\bGamma_{S}$ the principal submatrix of the covariance operator $\bGamma$ indexed by $S$. If $\bGamma$ is nonsingular, each such submatrix is invertible.

\subsection{A general minimax lower bound}
As is standard, an upper bound is obtained by exhibiting a test $f$ and then upper-bounding its risk---either \eqref{risk-known-f} or \eqref{risk-unknown-f} according to whether $\bGamma$ is known or unknown.
In order to derive a lower bound for the minimax risk, we follow the standard
argument of 
choosing a prior distribution on the class of alternatives and then lower-bounding the minimax risk with the resulting \emph{average risk}.
When $\bGamma$ is known, this leads us to select a prior on $\cC$, denoted by $\nu$, and consider
\beq \label{risk-nu-gamma}
\bar{R}_{\nu,\bGamma}(f) = \PROB_0\{f(\X)=1\} +  \sum_{S\in \cC} \nu(S) \PROB_{S,\bGamma}\{f(\X)=0\} \quad \text{and} \quad \bar{R}^*_{\nu,\bGamma} = \inf_f \bar{R}_{\nu,\bGamma}(f)~.
\eeq
The latter is the \emph{Bayes risk} associated with $\nu$.
By placing a prior on the class of alternative distributions, the alternative hypothesis becomes effectively simple (as opposed to composite).  The advantage of this is that the optimal 
test may be determined explicitly. Indeed, the Neyman-Pearson fundamental lemma implies that the likelihood ratio test $f_{\nu,\bGamma}^*(x) = \IND{L_{\nu,\bGamma}(x) > 1}$, with 
\[L_{\nu,\bGamma} = \sum_{S\in \cC} \nu(S) \frac{{\rm d}\PROB_{S,\bGamma}}{{\rm d}\PROB_0}~,\] 
minimizes the average risk.
In most of the paper, $\nu$ will be chosen as the uniform distribution on the class $\cC$.  In this because the sets in $\cC$ play almost the same role (although not exactly because of boundary effects).

When $\bGamma$ is only known to belong to some class $\mathfrak{G}$ we also need to choose a prior on $\mathfrak{G}$, which we denote by $\pi$, leading to
\beq \label{risk-nu-pi}
\bar{R}_{\nu,\pi}(f) = \PROB_0\{f(\X)=1\} + \sum_{S\in \cC} \nu(S) \int \PROB_{S,\bGamma}\{f(\X)=0\} \pi({\rm d}\bGamma) \quad \text{and} \quad \bar{R}^*_{\nu,\pi} = \inf_f \bar{R}_{\nu,\pi}(f)~.
\eeq
In this case, the likelihood ratio test becomes $f_{\nu,\pi}^*(x) = \IND{L_{\nu,\pi}(x) > 1}$, where 
\[L_{\nu,\pi} = \sum_{S\in \cC} \nu(S) \frac{{\rm d}\PROB_{S,\pi}}{{\rm d}\PROB_0}~, \quad \PROB_{S,\pi} = \int \PROB_{S,\bGamma} \pi({\rm d}\bGamma)~,\] 
minimizes the average risk.

In both cases, we then proceed to bound the second moment of the resulting likelihood ratio under the null.  Indeed, in a general setting, if $L$ is the likelihood ratio for $\P_0$ versus $\P_1$ and $R$ denotes its risk, then \cite[Problem 3.10]{TSH} 
\beq \label{LR-risk-general}
R = 1 - \frac12 \E_0 | L(X) - 1 | \geq 1 - \frac12 \sqrt{\E_0[L(X)^2] -1} \ ,
\eeq
where the inequality follows by the Cauchy-Schwarz inequality.

\begin{rem}
Working with the minimax risk (as we do here) allows us to bypass making an explicit choice of prior, although one such choice is eventually made when deriving a lower bound. 
Another advantage is that the minimax risk is monotone
with respect to the class $\cC$ in the sense that if $\cC'\subset
\cC$, then the minimax risk corresponding to $\cC'$ is at most as
large as that corresponding to $\cC$. This monotonicity does not necessarily hold for the Bayes risk.  See \cite{combin} for a discussion in the context of the detection-of-means problem.
\end{rem}

We now state a general minimax lower bound.  
(Recall that all the proofs are in \secref{proofs}.)
Although the result is stated for a class $\cC$ of disjoint subsets, using the monotonicity of the minimax risk, the result can be used to derive lower bounds in more general settings.  It is particularly useful in the context of detecting blob-like anomalous regions in the lattice.
(The same general approach is also fruitful in the detection-of-means setting.)
We emphasize that this result is quite straightforward given the work flow outlined above.  The technical difficulties will come with its application to the context that interest us here, which will necessitate a good control of \eqref{V_S} below.

Recall the definition \eqref{risk-nu-pi}.

\begin{prp}\label{prp:non_overlap_parametric}
Let $\{\bGamma(\phi): \phi\in \Phi\}$ be a class of nonsingular
covariance operators and let $\cC$ be a class of disjoint subsets of $\cV$.
Put the uniform prior $\nu$ on $\cC$ and let $\pi$ be a prior on $\Phi$. 
Then 
\[
\bar{R}^*_{\nu,\pi} \ge 1 - \frac{1}{2|\cC| } \Big(\sum_{S\in \cC}
V_S\Big)^{1/2}\ , 
\]
where
\beq \label{V_S}
V_S := \EXP_{\pi} \left[\left(\frac{\det(\bGamma_{S}^{-1}(\phi_1))\det(\bGamma_{S}^{-1}(\phi_2))}{\det(\bGamma_{S}^{-1}(\phi_1) + \bGamma_{S}^{-1}(\phi_2) -\bI_S)}\right)^{1/2} \right]\ ,
\eeq
and the expected value is with respect to $\phi_1,\phi_2$ drawn i.i.d.~from the distribution $\pi$. 
\end{prp}

\section{Known covariance} \label{sec:known}

We start with the case where the covariance operator $\bGamma$ is known.  Although this setting is of less practical importance, as this operator is rarely known in applications, we treat this case first for pedagogical reasons and also to contrast with the much more complex setting where the operator is unknown, treated later on.

\subsection{Lower bound} \label{sec:lower-known}

Recall the definition of the minimax risk \eqref{risk-known} and the average risk \eqref{risk-nu-gamma}.  
(Henceforth, to lighten the notation, we replace subscripts in $\bGamma(\phi)$ with subscripts in $\phi$.)  
For any prior $\nu$ on $\cC$, the minimax risk is at least as large as the $\nu$-average risk, $R^*_{\cC,\phi} \ge \bar{R}^*_{\nu,\phi}$, and the following corollary of Proposition~\ref{prp:non_overlap_parametric} provides a lower bound on the latter.

\begin{cor}\label{cor:non_overlap_GMRF}
Let $\cC$ be a class of disjoint subsets of $\cV$ and fix $\phi\in
\Phi_h$ satisfying $\|\phi\|_1<1/2$. 
Then, letting $\nu$ denote the uniform prior over $\cC$, we have
\beq\label{eq:lower_R*_GMRF1}
\bar{R}^*_{\nu,\phi} \geq  1 - \frac{1}{2|\cC| } \left[\sum_{S\in \cC}
  \exp\left(\frac{10 |S|\|\phi\|_2^2}{1- 2\|\phi\|_1}\right) \right]^{1/2}\ .
\eeq
\end{cor}

%The proof is in \secref{proof-cor-non_overlap_GMRF}.
In particular, the corollary implies that, for any fixed $a\in (0,1)$, $R^*_{\cC,\phi} \ge 1-a$ as soon as 
\beq\label{eq:lower_R*_GMRF1-2}
\frac{\|\phi\|_2^2}{1-2\|\phi\|_1} \leq  \min_{S\in \cC}\frac{\log\left(4|\cC|/a^2\right)}{10|S| }~.
\eeq
Furthermore, the hypotheses merge asymptotically (i.e., $R^*_{\cC,\phi} \to 1$) when
\beq \label{nonoverlap3}
 \log(|\cC|) - \frac{10\|\phi\|_2^2}{1- 2\|\phi\|_1}\max_{S \in \cC} |S| \to \infty~. 
\eeq

\begin{rem}
The condition $\|\phi\|_1<1/2$ in \corref{non_overlap_GMRF} is technical and likely an artifice of our proof method. This condition arises from the term $\det^{-1/2}(2\bGamma_S(\phi)-\bI_S)$ in $V_S$ in \eqref{V_S}.  For this determinant to be positive, the smallest eigenvalue of $\bGamma_S(\phi)$ has to be larger than $1/2$, which in turn is enforced by  $\|\phi\|_1<1/2$.  In order to remove, or at least improve on this constraint, we would need to adopt a more subtle approach than applying the Cauchy-Schwarz inequality in \eqref{LR-risk-general}.  We did not pursue this as typically one is interested in situations where $\phi$ is small --- see, for example, how the result is applied in \secref{cubes}. 
\end{rem}

\subsection{Upper bound: the generalized likelihood ratio test} \label{sec:glrt}

When the covariance operator $\bGamma(\phi)$ is known, the generalized likelihood ratio test rejects the null hypothesis for large values of 
\[
\max_{S \in \cC} \ X_S^\top (\bI_S - \bGamma_S^{-1}(\phi)) X_S~.
\]
We use instead the statistic
\beq \label{U}
U(X) = \max_{S \in \cC} \ \frac{X_S^\top (\bI_S - \bGamma_S^{-1}(\phi)) X_S - \tr(\bI_S - \bGamma_S^{-1}(\phi))}{\|\bI_S - \bGamma_S^{-1}(\phi)\|_F\sqrt{\log(|\cC|)}+ \|\bI_S - \bGamma_S^{-1}(\phi)\|\log(|\cC|)}~,
\eeq
which is based on the centering and normalization the statistics $X_S^\top (\bI_S - \bGamma_S^{-1}(\phi)) X_S$ where $S\in \cC$.

In the following result, we implicitly assume that $|\cC| \to \infty$, which is  the most interesting case.

\begin{prp} \label{prp:glrt}
Assume that $\phi\in \Phi_h$ satisfies  $\|\phi\|_1 \le \eta < 1$ and that $|S^h|\geq |S|/2$. 
The test $f(x)=\IND{U(x) > 4}$ has risk $R_{\cC,\phi}(f) \le 2/|\cC|$ when
\beq \label{glrt4}
\|\phi\|_2^2 \min_{S \in \cC} |S| \ge C_0 \log(|\cC|)\ ,
\eeq
where $C_0 > 0$ only depends on the dimension $d$ of the lattice and $\eta$. 
\end{prp}

Comparing with Condition \eqref{nonoverlap3}, we see that condition \eqref{glrt4} matches (up to constants) the minimax lower bound, so that (at least when $\|\phi\|_1 < 1/2$) the normalized generalized likelihood ratio test based on \eqref{U} is asymptotically minimax up to a multiplicative constant. 
The $\ell_1$-norm  $\|\phi\|_1$ arises in the proof of Corollary \ref{cor:non_overlap_GMRF} when bounding the largest eigenvalue of $\bGamma(\phi)$ (see Lemma~\ref{lem:spectrum_gamma}).

\section{Unknown covariance} \label{sec:unknown}

We now consider the case where the covariance operator $\bGamma(\phi)$ of the anomalous Gaussian Markov random field is unknown.  We therefore start by defining a class of covariance operators via a class of vectors $\phi$.
Given a positive integer $h>0$ and some $r>0$, define 
\beq\label{eq:definition_phi_h_r}
\Phi_h(r) := \{\phi\in \Phi_h,\ \|\phi\|_2 \ge r\}\ ,
\eeq
and let 
\beq\label{eq:G}
\mathfrak{G}(h,r) := \{\bGamma(\phi) : \phi \in \Phi_h(r)\}~,
\eeq
which is the class of covariance operators corresponding to stationary Gaussian Markov Random Fields with parameter in the class \eqref{eq:definition_phi_h_r}.

\subsection{Lower bound} \label{sec:lower-unknown}

The theorem below establishes a lower bound for the risk following the approach outlined in \secref{prelim}, which is based on the choice of a suitable prior $\pi$ on $\Phi_h$, defined as follows. 
By symmetry of the elements of $\Phi_h$, one can fix a sublattice
$\bbN'_h$  of size $|\bbN_h|/2$ such that any $\phi\in\Phi_h$ is
uniquely defined (via symmetry) by its restriction to $\bbN'_h$. Choose the
distribution $\pi$ such that $\phi\sim \pi$ is the unique extension to
$\bbN_h$ of the random vector $ r|\bbN_h|^{-1/2}\xi$, where the
coordinates of the random vector $\xi$---indexed by $\bbN'_h$---are
i.i.d.\ Rademacher random variables (i.e., symmetric $\pm 1$-valued
random variables). Note that, if $r|\bbN_h|<1$, $\pi$ is acceptable   since it concentrates on the set
$\{\phi\in \Phi_h,\ \|\phi\|_2 = r\} \subset \Phi_h(r)$. 
%This is sufficient for our purposes as the testing problem becomes easier for larger values of $\|\phi\|_2$.
Recall the definition of the minimax risk \eqref{risk-unknown} and the average risk \eqref{risk-nu-pi}.  
%(Henceforth, to lighten the notation, we replace $\{\bGamma(\phi), \phi \in \Phi_h(r)\}$ in the subscript with $\Phi_h(r)$.)
As before, for any priors $\nu$ on $\cC$ and $\pi$ on $\Phi_h(r)$, the minimax risk is at least as large as the average risk with these priors, $R^*_{\cC,\mathfrak{G}(h,r)} \ge \bar{R}^*_{\nu,\pi}$, and the following (much more elaborate) corollary of Proposition~\ref{prp:non_overlap_parametric} provides a lower bound on the latter. 

\begin{thm}
\label{thrm3:non_overlap_GMRF}
There exists a constant $C_0>0$ such that the following holds. Let $\cC$
be a class of disjoint subsets of $\cV$ and let $\nu$ denote the uniform prior over $\cC$.  Let $a\in (0,1)$
and assume that the neighborhood size $|\bbN_h|$ satisfies
\beq\label{eq:condition_Nh}
|\bbN_h|\leq  \min_{S\in \cC}\left[\frac{|S|}{\log\left(|\cC|/a\right)}
\bigwedge |S|^{2/5}\log^{1/5}\left(|\cC|/a\right)
\bigwedge \left(\frac{|S|}{|\Delta_{2h}(S)|}\right)^2\log^{-1/6}\left(|\cC|/a\right)\right]\ .
\eeq
Then 
$
\bar{R}^*_{\nu,\pi} \ge 1-a
$
as soon as 
\beq\label{eq:condition_r2_lower}
r^2 \max_{S \in \cC} |S| \leq C_0 \left[\sqrt{|\bbN_h|\log\left(|\cC|/a\right)} \bigvee \log\left(|\cC|/a\right) \right]\ .
\eeq
\end{thm}

This bound is our main impossibility result. Its proof relies on a number auxiliary results for Gaussian Markov Random Fields (Section~\ref{sec:proofs_technique}) that may useful for other problems
of estimating Gaussian Markov Random Fields.
Notice that the second term in \eqref{eq:condition_r2_lower} is what appears in \eqref{eq:lower_R*_GMRF1-2}, which we saw arises in the case where the covariance is known.  In light of this fact, we may interpret the first term in \eqref{eq:condition_r2_lower} as the `price to pay' for adapting to an unknown covariance operator in the class of covariance operators of Gaussian Markov random fields with dependency radius $h$.

\subsection{Upper bound: a Fisher-type test} \label{sec:fisher}

We introduce a test whose performance essentially matches the minimax lower bound of Theorem \ref{thrm3:non_overlap_GMRF}.
Comparatively, the construction and analysis of this test is much more involved than that of the generalized likelihood ratio test of \secref{glrt}.

Let $F_i = (X_{i+v} : 1 \le |v|_{\infty} \le h)$, seen as a vector, and let $\bF_{S,h}$ be the matrix with row vectors $F_i, i \in S^h$. 
Also, let $X_{S,h} = (X_i : i \in S^h)$. Under the null hypothesis, each variable $X_i$ is independent of $F_i$, although $X_i$ is correlated with some $(F_j, j\neq i)$. Under the alternative hypothesis, there exists a subset $S$ and a vector $\phi\in \Phi_h$ such that
\beq\label{eq:conditional}
X_{S,h}= \bF_{S,h}\phi+ \epsilon_{S,h}\ ,
\eeq
where each component $\epsilon_i$ of $\epsilon_{S,h}$ is independent of the corresponding vector $F_{i}$, but the $\epsilon_i$'s are not necessarily independent. Equation \eqref{eq:conditional} is the so-called conditional autoregressive (CAR) representation of a Gaussian Markov random field \citep{MR1344683}. 
For Gaussian Markov random fields, the celebrated  pseudo-likelihood method~\citep{besag:1975} amounts to  estimating $\phi$ by taking least-squares in \eqref{eq:conditional}.

Returning to our testing problem, observe
that the null hypothesis is true if and only if all the parameters of the conditional expectation of $X_{S,h}$ given $\bF_{S,h}$ are zero. 
In analogy with the analysis-of-variance approach for testing whether the coefficients of a linear regression model are all zero, we consider a Fisher-type statistic
\beq\label{phi-stat}
T^*= \max_{S\in \cC}T_S\ , \quad \quad T_S := \frac{|S^h|\|\boldsymbol{\Pi}_{S,h}X_{S,h}\|_2^2}{\|X_{S,h}-\boldsymbol{\Pi}_{S,h}X_{S,h}\|_2^2}\ , 
\eeq
where $\boldsymbol{\Pi}_{S,h} := \bF_{S,h}(\bF_{S,h}^\top  \bF_{S,h})^{-1}\bF^\top _{S,h}$ is the orthogonal projection onto the column space of $\bF_{S,h}$. 
Since in the linear model \eqref{eq:conditional} the response vector $X_{S,h}$ is not independent of the design matrix $\bF_{S,h}$, the statistic $T_S$ does not follow an $F$-distribution. Nevertheless, we are able to control the deviations of $T^*$, both under null and alternative hypotheses, leading to the following performance bound.
Recall the definition \eqref{risk-unknown-f}.

\begin{thm}\label{thm:LS1}
There exist four positive constants $C_1,C_2,C_3,C_4$ depending only on $d$ such that the following holds. Assume that 
\beq\label{eq:condition_LS}
|\bbN_h|^4\vee |\bbN_h|^2 \log(|\cC|)  \leq C_1 \min_{S\in \cC}|S^h|\ .
\eeq
Fix $\alpha$ and $\beta$ in $(0,1)$ such that 
\beq \label{alpha_beta}
\log(\tfrac{1}{\alpha})\vee \log(\tfrac{1}{\beta})\leq C_2 \frac{\min_{S\in \cC}|S^h|}{|\bbN_h|^2\log(|\cC|)}\ .
\eeq
Then, under the null hypothesis,
\beq\label{eq:upper_TS_H0}
\P\left\{ T^* \geq   |\bbN_h|+ C_3\left[\sqrt{|\bbN_h|(\log(|\cC|)+ 1+\log(\alpha^{-1}) )}+ \log(|\cC|)+ \log(\alpha^{-1})\right] \right\} \leq \alpha \ ,
\eeq
while under the alternative,
\beq\label{eq:upper_TS_H1}
\P\left\{ T^* \geq |\bbN_h|  + C_4 \left[|S^h|\left(\|\phi\|_2^2\wedge \frac{1}{|\bbN^h|}\right)  - \sqrt{\bbN_h}(1+\log^{4}(\beta^{-1}))\right]\right\} \geq 1 - \beta\ .
\eeq
In particular, if $\alpha_n,\beta_n\to 0$ are arbitrary positive sequences,
then the test $f$ that rejects the null hypothesis if
\[
T^* \ge |\bbN_h| + C_3\left[\sqrt{|\bbN_h|(\log(|\cC|)+ 1+\log(\alpha_n^{-1}) )}+ \log(|\cC|)+ \log(\alpha_n^{-1})\right]
\]
satisfies $R_{\cC,\mathfrak{G}(h,r)}(f) \to 0$ 
as soon as
\beq\label{eq:power_LS}
r^2  > \frac{C_0}{\min_{S\in \cC}|S^h|}\left[\sqrt{|\bbN_h|\left(\log(|\cC|)+\log(\tfrac{1}{\alpha_n})+ \log^{8}(\tfrac{1}{\beta_n}) \right)}\bigvee \log(|\cC|)\bigvee  \log(\tfrac{1}{\alpha_n})\right]\ ,
\eeq 
where $C_0 > 0$ depends only on $d$. 
%Thus, $f$ asymptotically separates the two hypotheses.
\end{thm}

Comparing with the minimax lower bound established in Theorem \ref{thrm3:non_overlap_GMRF}, we see that this test is nearly optimal with respect to $h$, the size of the collection $|\cC|$, and the size $|S|$ of the anomalous region (under the alternative).

\section{Examples: cubes and blobs} \label{sec:cubes}

In this section we specialize our general results proved in the previous subsections to classes of cubes, and more generally, blobs. 

\subsection{Cubes}
Consider the problem of detecting an anomalous cube-shaped region.
Let $\ell \in \{1, \dots, m\}$ and assume that $m$ is an integer multiple of $\ell$ (for simplicity).  Let $\cC$ denote the class of all discrete hypercubes of side length $\ell$, that is, sets of the form
$S = \prod_{s=1}^d \{b_s,\ldots,b_s+\ell-1\}$, where $b_s\in \{1,\ldots,m+1-\ell\}$.
Each such hypercube $S \in \cC$ contains $|S| = k := \ell^d$ nodes, and the class is of size $|\cC| = (m-1-\ell)^d \le n$.

The lower bounds for the risk established in \corref{non_overlap_GMRF} 
and Theorem~\ref{thrm3:non_overlap_GMRF} are not directly applicable here since these results require subsets of the class $\cC$ to be disjoint.
However, they apply to any subclass $\cC' \subset \cC$ of disjoint subsets and, as mentioned in \secref{prelim}, any lower bound on the minimax risk over $\cC'$ applies to the minimax risk over $\cC$.  A natural choice for $\cC'$ here is that of all cubes of the form 
$S = \prod_{s=1}^d \{a_s \ell +1, \dots, (a_s+1) \ell\}$, 
where $a_s \in \{0, \dots, m/\ell-1\}$. 
Note that $|\cC'|=(m/\ell)^d= n/k$.

\medskip\noindent {\bf $h$ bounded.}
Consider first the case where the radius $h$ of the neighborhood is bounded.  We may apply \corref{non_overlap_GMRF} to get 
\[
R^*_{\cC,\phi} \geq 1- \frac{k^{1/2}}{2n^{1/2}}\exp\left\{\frac{5k\|\phi\|_2^2}{1-2\|\phi\|_1}\right\}\ .\]
For a given $r>0$ satisfying $2|\bbN_h|r\leq 1$, we can choose a parameter $\phi$ constant over $\bbN_h$ such that $\|\phi\|_2=r$ and $\|\phi\|_1=r\sqrt{(2h+1)^d-1}$.  Since $R^*_{\cC,\mathfrak{G}(h,r)}\geq R^*_{\cC,\phi}$, 
 we thus have $R^*_{\cC,\mathfrak{G}(h,r)} \to 1$ when $n \to \infty$, if $(k,\phi) = (k(n), \phi(n))$ satisfies $\log(n)\ll k\ll n$ and  $r^2 \le \log(n/k)/(11 k)$.
Comparing with the performance of the Fisher test of \secref{fisher}, in this particular case, Condition \eqref{eq:condition_LS} is met, and letting $\alpha = \alpha(n) \to 0$ and $\beta = \beta(n) \to 0$ slowly, we conclude from \eqref{eq:power_LS} that this test (denoted $f$) has risk $R_{\cC,\mathfrak{G}(h,r)}(f) \to 0$ when $r^2 \ge C_0 \log(n)/k$ for some constant $C_0$.  
Thus, in this setting, the Fisher test, without knowledge of $\phi$, achieves the correct detection rate as long as $k \le n^b$ for some fixed $b < 1$.

\medskip\noindent {\bf $h$ unbounded.}
When $h$ is unbounded, we obtain a sharper bound by using Theorem~\ref{thrm3:non_overlap_GMRF} instead of \corref{non_overlap_GMRF}.  Specialized to the current setting, we derive the following.

\begin{cor}\label{cor:lower_hypercube}
There exist two positive constants $C_1$ and $C_2$ depending only on $d$ such that the following holds.
Assume that the neighborhood size $h$ is small enough that
\beq\label{eq:condition_neigbhorhood}
|\bbN_h|\leq C_1\left[  \frac{k}{\log^{1\vee (d/2)}\left(\frac{n}{k}\right)}\bigwedge k^{2/5}\log^{1/5}\left(\frac{n}{k}\right)\bigwedge d^{-\frac{2d}{d+2}} k^{\frac{2}{d+2}}\log^{\frac{d}{3d+6}}\left(\frac{n}{k}\right)\right]\ .
\eeq
Then the minimax risk tends to one when $n \to \infty$ as soon as $(k,h,r)=(k(n),h(n),r(n))$ satisfies $n/k\rightarrow \infty$ and
\beq\label{eq:powerless_gmrf}
r^2 \le C_2 \left[\frac{\log(\tfrac{n}{k})}{k}\bigvee\frac{\sqrt{|\bbN_{h}|\log(\tfrac nk)}}{k}\right] \ .
\eeq
\end{cor}

Note that, in the case of a square neighborhood, $|\bbN_h| = (2h+1)^d -1$.
Comparing with the performance of the Fisher test, in this particular case, Condition \eqref{eq:condition_LS} is equivalent to $|\bbN_h| \le C_0 \big(k^{1/4} \wedge \sqrt{k/\log(n)}\big)$ for some constant $C_0$.  When $k$ is polynomial in $n$, this condition is stronger than Condition \eqref{eq:condition_neigbhorhood} unless $d \le 5$.  In any case, assuming $h$ is small enough that both \eqref{eq:condition_LS} and \eqref{eq:condition_neigbhorhood} hold, and letting $\alpha = \alpha(n) \to 0$ and $\beta = \beta(n) \to 0$ slowly, we conclude from \eqref{eq:power_LS} that the Fisher test has risk $R_{\cC,\mathfrak{G}(h,r)}$ tending to zero  when 
\[r^2 \ge C_0 \left[\frac{\log(n)}{k}\bigvee\frac{\sqrt{|\bbN_{h}|\log(n)}}{k}\right],\] 
for some large-enough constant $C_0 > 0$, matching the lower bound \eqref{eq:powerless_gmrf} up to a multiplicative constant as long as $k \le n^b$ for some fixed $b < 1$.

\medskip
In conclusion, whether $h$ is fixed or unbounded but growing slowly enough, the Fisher test achieves a risk matching the lower bound up to a multiplicative constant.

\subsection{Blobs}
So far, we only considered hypercubes, but our results generalize immediately to much larger classes of blob-like regions.  Here, we follow the same strategy used in the detection-of-means setting, for example, in \cite{MGD,cluster,MR2589191}.

Fix two positive integers $\ell_\circ \le \ell^\circ$ and let $\cC$ be a class of subsets $S$ such that there are hypercubes $S_\circ$ and $S^\circ$, of respective side lengths $\ell_\circ$ and $\ell^\circ$, such that $S_\circ \subset S \subset S^\circ$.  Letting $\cC_\circ$ and $\cC^\circ$ denote the classes of hypercubes of side lengths $\ell_\circ$ and $\ell^\circ$, respectively, our lower bound for the worst-case risk associated with the class $\cC^\circ$ obtained from \corref{lower_hypercube} applies directly to $\cC$---although not completely obvious, this follows from our analysis---while scanning over $\cC_\circ$ in the Fisher test yields the performance stated above for the class of cubes. In particular, if $\ell_\circ/\ell^\circ$ remains bounded away from $0$, the problem of detecting a region in $\cC$ is of difficulty comparable to detecting a hypercube in $\cC_\circ$ or $\cC^\circ$.

When the size of the anomalous region $k$ is unknown, meaning that the class $\cC$ of interest includes regions of different sizes, we can simply scan over dyadic hypercubes as done in the first step of the multiscale method of \cite{MGD}.  
This does not change the rate as there are less than $2n$ dyadic hypercubes.
See also \cite{cluster}.
 
We note that when $\ell_\circ/\ell^\circ = o(1)$, scanning over hypercubes may not be very powerful.
For example, for ``convex" sets, meaning when 
\[\cC = \Big\{S = K \cap \cV: K \subset \bbR^d \text{ convex}, |K \cap \cV| = k\Big\}~,\] 
it is more appropriate to scan over ellipsoids due to John's ellipsoid theorem \citep{MR0030135}, which implies that for each convex set $K \subset \bbR^d$, there is an ellipsoid $E \subset K$ such that ${\rm vol}(E) \ge d^{-d} {\rm vol}(K)$.
For the case where $d=2$ and the detection-of-means problem, \cite{MR2589191}---expanding on ideas proposed in \cite{MGD}---scan over parallelograms, which can be done faster than scanning over ellipses.

Finally, we mention that what we said in this section may apply to other types of regular lattices, and also to lattice-like graphs such as typical realizations of a random geometric graph.  See \cite{cluster,MR2604703} for detailed treatments in the detection-of-means setting.

\section{Discussion} \label{sec:discussion}

We provided lower bounds and proposed near-optimal procedures for testing for the presence of a piece of a Gaussian Markov random field.
These results constitute some of the first mathematical results for the problem of detecting a textured object in a noisy image.
%It also extends our own previous work \citep{correlation-detect,multidim} on the detection of correlations to setting where the correlations have some nontrivial structure.
We leave open some questions and generalization of interest.

{\em More refined results.}
We leave behind the delicate and interesting problem of finding the exact detection rates, with tight multiplicative constants.  
This is particularly appealing for simple settings such as finding an interval of an autoregressive process, as described in \secref{intro-tseries}.
Our proof techniques, despite their complexity, are not sufficiently refined to get such sharp bounds.  
We already know that, in the detection-of-means setting, bounding the variance of the likelihood ratio does not yield the right constant.
The variant which consists of bounding the first two moments of a carefully truncated likelihood ratio, possibly pioneered in \cite{Ingster99}, is applicable here, but the calculations are quite complicated and we leave them for future research.

{\em Texture over texture.}
Throughout the paper we assumed that the background is Gaussian white noise. 
This is not essential, but makes the narrative and results more accessible.
A more general, and also more realistic setting, would be that of detecting a region where the dependency structure is markedly different from the remainder of the image.  This setting has been studied in the context of time series, for example, in some of the references given in \secref{intro-tseries}.  However, we are not aware of existing theoretical results in higher-dimensional settings such as in images.

{\em Other dependency structures.}
We focused on Markov random fields with limited neighborhood range (quantified by $h$ earlier in the paper).  This is a natural first step, particularly since these are popular models for time series and textures.
However, one could envision studying other dependency structures, such as short-range dependency, defined in \cite{MR2379935} as situations where the covariances are summable in the following sense
\[
\sup_{i\in \cV_\infty} \sum_{j \in \cV_\infty \setminus i} |\bGamma_{i,j}| < \infty\ .
\]

\section{Proofs} \label{sec:proofs}

\subsection{Proof of \prpref{non_overlap_parametric}}

The Bayes risk is achieved by the likelihood ratio
test $f_{\nu,\pi}^*(x)=\IND{L_{\nu,\pi}(x)>1}$ where
\[
L_{\nu,\pi}(x)= \frac{1}{|\cC|} \sum_{S\in \cC} L_S(x)~, \text{ with} \quad
L_S(x) = \int \frac{ {\rm d}\PROB_{S,\bGamma(\phi)}(x) } {{\rm d} \PROB_0(x)} \pi({\rm d}\phi)~.
\]
In our Gaussian model,
\beq \label{L_S}
L_S(x) = \EXP_{\pi}\left[ \exp\left(\tfrac12 x_S^\top (\bI_S - \bGamma_{S}^{-1}(\phi)) x_S - \tfrac12 \log \det(\bGamma_{S}(\phi))\right)\right]~,
\eeq
where the expectation is taken with respect to the random draw of
$\phi \sim \pi$.
Then, by \eqref{LR-risk-general},
\beq \label{LR-risk}
\bar{R}^*_{\nu,\pi} = 1 - \frac12 \E_0 | L_{\nu,\pi}(\X) - 1 | \geq 1 - \frac12 \sqrt{\E_0[L_{\nu,\pi}(\X)^2] -1} \ .
\eeq
(Recall that $\EXP_0$ stands for expectation with respect to the standard normal random vector $\X$.)

We proceed to bound the second moment of the likelihood ratio under the null hypothesis.  Summing over $S,T \in \cC$, we have
\beqn
&& \hspace{-.3in} \E_0[L_{\nu,\pi}(\X)^2] \\
&=& \frac1{|\cC|^2} \sum_{S,T\in \cC} \E_0[L_S(\X) L_T(\X)] \\
&=& \frac1{|\cC|^2} \sum_{S \ne T} \E_0[ L_S(\X)] \E_0 [L_T(\X)] +
\frac1{|\cC|^2} \sum_{S\in \cC} \E_0 [L_S^2(\X)] \\
%& & \text{(since $S \ne T$ are disjoint, and therefore $L_S(\X)$ and
%  $L_T(\X)$ are independent)} \\
&=& \frac{|\cC|-1}{|\cC|} + \frac1{|\cC|^2} \sum_{S\in\cC} \E_0 \EXP_{\pi}\Big[\exp\Big(X_S^\top \big(\bI_S - \tfrac12 \bGamma_{S}^{-1}(\phi_1) - \tfrac12  \bGamma_{S}^{-1}(\phi_2) \big) X_S \\ 
&& \qquad \qquad \qquad \qquad \qquad \qquad 
- \tfrac{1}{2}\log \det(\bGamma_{S}(\phi_1))- \tfrac{1}{2}\log
\det(\bGamma_{S}(\phi_2)) \Big)\Big] \\
%& & \text{(since $\E_0 [L_S(\X)] = 1$ for all $S\in \cC$)} \\
& \leq  & 1+  \frac1{|\cC|^2} \sum_{S}\EXP_{\pi}\Big[ \exp\Big( - \tfrac12 \log \det(\bGamma_{S}^{-1}(\phi_1)+\bGamma_{S}^{-1}(\phi_2) - \bI_S) - \tfrac{1}{2}\log \det(\bGamma_{S}(\phi_1)\bGamma_{S}(\phi_2)) \Big)\Big] \\
& = & 1 + \frac{1}{|\cC|^2 } \sum_{S} V_S  \ ,
\eeqn
where in the second equality we used the fact that $S \ne T$ are disjoint, and therefore $L_S(\X)$ and $L_T(\X)$ are independent, and in the third we used the fact that $\E_0 [L_S(\X)] = 1$ for all $S\in \cC$.

\subsection{Deviation inequalities}

Here we collect a few more-or-less standard inequalities that we need in the proofs.
We  start with the following standard tail bounds for Gaussian quadratic
forms.  See, e.g., Example 2.12 and Exercise 2.9 in \cite{BoLuMa13}.
  
\begin{lem} \label{lem:normal-quad}
Let $Z$ be a standard normal vector in $\bbR^d$ and let $\bR$ be a symmetric $d\times d$ matrix. Then
\[
\pr{Z^\top \bR Z - \tr(\bR) \geq  2\|\bR\|_F \sqrt{t} + 2 \|\bR\|t \, } \le e^{-t}, \quad \forall t \ge 0~.
\]
Furthermore, if the matrix $\bR$ is positive semidefinite, then
\[
\pr{Z^\top \bR Z - \tr(\bR) \leq  -2\|\bR\|_F \sqrt{t} \, } \le e^{-t}, \quad \forall t \ge 0.
\]
\end{lem}

\begin{lem}\label{lem:chaos_4}
There exists a positive constant $C$ such that the following holds. 
For any  Gaussian chaos $Z$ up to order $4$ and any $t>0$,
\[\P\left\{|Z-\E[Z]|\geq C\Var^{1/2}(Z)t^2 \right\}\leq e^{-t}\ .\]
\end{lem}

\begin{proof}
 This deviation inequality is a consequence of the hypercontractivity of Gaussian chaos. More precisely, Theorem 3.2.10 and Corollary 3.2.6 in \cite{MR1666908} state that
\[\E\exp\left[\left(\frac{Z-\E[Z]}{C\Var^{1/2}(Z)}\right)^{1/2}\right]\leq 2\ ,\]
where $C$ is a numerical constant. Then, we apply Markov inequality to prove the lemma.
\end{proof}

\begin{lem}\label{lem:chaos}
There exists a positive constant $C$ such that the following holds.
Let $F$ be a compact set of symmetric $r\times r$ matrices and let $Y\sim \cN(0,\bI_r)$. For any $t>0$, the random variable  $Z:= \sup_{\bR \in F} \tr\left[\bR YY^\top \right]$ satisfies 
\beq \label{eq:chaos_concentration}
\mathbb{P}\{Z \geq \mathbb{E}(Z) + t\} \leq \exp\left(-C\left( \frac{t^2}{\mathbb{E}(W)}\bigwedge \frac{t}{B}\right)\right),
\eeq
where  $W  := \sup_{\bR \in F} \tr(\bR YY^\top \bR)$ and  $B  :=  \sup_{\bR \in F}\|\bR\|$.
\end{lem}
A slight variation of this result where $Z$ is replaced by $\sup_{\bR \in F} \tr\left[\bR (YY^\top -\bI_r)\right]$ is proved in \cite{MR2662346} using the exponential Efron-Stein inequalities of~\cite{MR2123200}. Their arguments straightforwardly adapt to Lemma~\ref{lem:chaos}.

\begin{lem}[\cite{MR1863696}] \label{lem:concentration_vp_wishart}
Let $\bW$ be a standard Wishart matrix with parameters $(n,d)$ satisfying $n>d$. Then for any number $0<x<1$, 
\begin{eqnarray}
\mathbb{P}\left\{\lambda^{\rm max}(\bW) \geq
n\left(1+\sqrt{d/n}+\sqrt{2x/n}\right)^2 \right \} & \leq &e^{-x}\ 
,\nonumber	\\
%\label{majoration_wishart_sous_gaussienne}
\mathbb{P}\left\{\lambda^{\rm min}(\bW) \leq
n\left(1-\sqrt{d/n}-\sqrt{2x/n}\right)_+^2 \right\} & \leq
&e^{-x}\ . \nonumber
\end{eqnarray}
\end{lem}

\subsection{Auxiliary results for Gaussian Markov random fields on the lattice}\label{sec:proofs_technique}

He we gather some technical tools and proofs for Gaussian Markov
random fields on the lattice.
Recall the notation introduced in Section~\ref{sec:prelim_gmrf}.

\begin{lem}
\label{lem:spectrum_gamma}
For any positive integer $h$ and $\phi\in \Phi_{h}$ with $\|\phi\|_1 <1$, we have that
if $\lambda$ is an eigenvalue of the covariance operator
$\bGamma(\phi)$, then
\[
\frac{\sigma_{\phi}^2}{1+\|\phi\|_1} \le \lambda \le
\frac{\sigma_{\phi}^2}{1-\|\phi\|_1}~.
\]
Also, we have
\beq\label{eq:lb_sigma}
\frac{\|\phi\|_2^2}{1+ \|\phi\|_1}\leq \frac{1 - \sigma^{2}_{\phi}}{\sigma^2_{\phi}}\leq \frac{\|\phi\|_2^2}{1-\|\phi\|_1}\ \quad \quad \text{ and }\quad \quad 1- \|\phi\|_1 \le  \sigma_{\phi}^2 \le 1~.
\eeq
\end{lem}
\begin{proof}
Recall that $\|\cdot\|$ denotes the $\ell^2 \to \ell^2$ operator norm.
First note that by the definition of $\phi$,
$\sigma_\phi^2 \bGamma^{-1}(\phi)- \bI = (\phi_{i-j})_{i,j \in \bbZ^d}$, and therefore 
\beq \label{diag_dominant}
\|\sigma_\phi^2 \bGamma^{-1}(\phi) - \bI\| \le  \|\phi\|_1~,
\eeq
where whe used the bound $\|\bA\|\leq \sup_{i\in \bbZ^d}\sum_{j\in \bbZ^d}|\bA_{ij}|$. 
This implies that the largest eigenvalue of $\bGamma(\phi)$ is bounded
by $\sigma^2_{\phi}/(1-\|\phi\|_1)$ if $\|\phi\|_1<1$ and that the
smallest eigenvalue of $\bGamma(\phi)$ is at least $\sigma^2_{\phi}/(1+\|\phi\|_1)$.
Considering the conditional regression of $Y_i$ given $Y_{-i}$
mentioned above, that is, 
\[Y_i= -\sum_{1\leq |j|_{\infty}\leq h}\phi_j Y_{i+j}+\epsilon_i
\]
(with $\epsilon_i$ being standard normal independent of the $Y_j$ for
$j\neq i$)
and taking the variance of both sides, we obtain 
\[1-\sigma^2_{\phi}= \Var\left[ \sum_{1<|j|_{\infty}\leq h}\phi_j Y_{i+j}\right] = \phi^\top \bGamma(\phi) \phi \le \|\bGamma(\phi)\| \|\phi\|_2^2\leq \frac{\|\phi\|_2^2}{1-\|\phi\|_1}\sigma^2_{\phi}~,
\]
and therefore 
\beqn
1 - \sigma^{2}_{\phi}\leq \frac{\|\phi\|_2^2}{1-\|\phi\|_1}\sigma^{2}_{\phi}\ .
\eeqn
Rearranging this inequality and using the fact that $\|\phi\|_2^2 \le
\|\phi\|_1^2 \le \|\phi\|_1$, we conclude that $\sigma^{2}_{\phi}\geq 1-\|\phi\|_1$. The remaining bound $\frac{\|\phi\|_2^2}{1+ \|\phi\|_1}\leq \frac{1 - \sigma^{2}_{\phi}}{\sigma^2_{\phi}}$ is obtained similarly.
\end{proof}

Recall that for any $v\in \mathbb{Z}^d$, $\gamma_v$ is the correlation between $Y_i$ and $Y_{i+v}$ and is therefore equal to $\bGamma_{i, i+v}$. This definition does not depend on the node $i$ since $\bGamma$ is the covariance of a stationary process.

\begin{lem}\label{lem:conditional_variance}
For any $h$ and any $\phi\in \Phi_{h}$, let $Y\sim \cN(0,\bGamma(\phi))$. As long as $\|\phi\|_1<1$, the $l_2$ norm of the correlations satisfies
\beq\label{eq:upper_norm_correlation}
\sum_{v\neq 0}\gamma_v^2\leq   \frac{\|\phi\|^2_2}{(1-\|\phi\|_1)^2}+ \left(\frac{\|\phi\|_2^2\sigma_{\phi}^2}{(1-\|\phi\|_1)^2}\right)^2
\eeq
\end{lem}

\begin{proof}
In order to compute $\|\gamma\|^2_2$, we use the spectral density of $Y$ defined by 
\[f(\omega_1,\ldots,\omega_d)=\frac{1}{(2\pi)^d}\sum_{(v_1,\ldots,v_d)\in \mathbb{Z}^d}\gamma_{v_1,\ldots v_d}\exp\left(\iota\sum_{i=1}^dv_i\omega_i\right)\ , \quad (\omega_1,\ldots,\omega_d)\in (-\pi,\pi]^d\ .\]
Following \cite[][Sect.1.3]{MR1344683} or
\cite[][Sect.2.6.5]{MR2130347}, we express the spectral density in
terms of $\phi$ and $\sigma_{\phi}^2$:
%\eqref{markov-lattice}
\[\frac{1}{f(\omega_1,\ldots, \omega_d)}= \frac{(2\pi)^d}{\sigma_{\phi}^2}\left[1- \sum_{v, 1\leq |v|_{\infty}\leq h\in \mathbb{Z}^d}\phi_{v}e^{\iota \langle v,\omega\rangle}\right]\ ,\]
where $\langle \cdot , \cdot\rangle$ denotes the scalar product in $\mathbb{R}^d$. As a consequence, 
\[|f(\omega_1,\ldots, \omega_d)|\leq \sigma_{\phi}^2[(2\pi)^d(1-\|\phi\|_1)]^{-1}~.\] Relying on Parseval formula, we conclude 
\beqn
\sum_{v\neq 0}\gamma_v^2&=& (2\pi)^d\int_{[-\pi;\pi]^d}\Big[f(\omega_1,\ldots,\omega_d) -\frac{1}{(2\pi)^d}\Big]^2d\omega_1\ldots d\omega_d\\
&\leq  & \frac{\sigma_{\phi}^4}{(2\pi)^d(1-\|\phi\|_1)^2}\int_{[-\pi;\pi]^d}\Big|\frac{1}{(2\pi)^df(\omega_1,\ldots, \omega_d)} -1\Big|^2d\omega_1\ldots d\omega_d\\
&\leq & \frac{\sigma_{\phi}^4}{(2\pi)^d(1-\|\phi\|_1)^2}\int_{[-\pi;\pi]^d}\Big|\frac{1}{\sigma_{\phi}^2} - 1 - \frac{1}{\sigma_{\phi}^2}\sum_{v, 1\leq |v|_{\infty}\leq h\in \mathbb{Z}^d} \phi_{v}e^{\iota \langle v,\omega\rangle}\Big|^2d\omega_1\ldots d\omega_d\\
&\leq & \frac{\sigma_{\phi}^4}{(2\pi)^d(1-\|\phi\|_1)^2}\left[(2\pi)^d\left(\frac{1}{\sigma^2_{\phi}}-1\right)^2+\sum_{v, 1\leq |v|_{\infty}} \frac{(2\pi)^d\phi_v^2}{\sigma_{\phi}^4} \right]\\
& \leq & \left(\frac{1-\sigma_{\phi}^2}{1-\|\phi\|_1}\right)^2+\frac{\|\phi\|^2_2}{(1-\|\phi\|_1)^2}\\
& \leq & \left(\frac{\|\phi\|_2^2\sigma_{\phi}^2}{(1-\|\phi\|_1)^2}\right)^2+  \frac{\|\phi\|^2_2}{(1-\|\phi\|_1)^2}\ , 
\eeqn
where we used \eqref{eq:lb_sigma} in the last line. 
\end{proof}

\begin{lem}[Conditional representation]\label{lem:covariance_residuals}
For any $h$ and any $\phi\in \Phi_{h}$, let $Y\sim
\cN(0,\bGamma(\phi))$. Then for any $i\in \bbZ^d$, the random variable
$\epsilon_i$ defined by the conditional regression $Y_i=\sum_{v\in
  \bbN_h}\phi_vY_{i+v} +\epsilon_i$ satisfies that
\begin{enumerate}
 \item $\epsilon_i$ is independent of all $X_j\ , j\neq i$ and $\Cov(\epsilon_i,X_i)=\Var(\epsilon_i)=\sigma_{\phi}^2$.
\item For any $i\neq j$, $\Cov(\epsilon_i,\epsilon_j)=-\phi_{i-j}\sigma_{\phi}^2$ if $|i-j|_{\infty}\leq h$ and 0 otherwise.
\end{enumerate}
\end{lem}

\begin{proof}
%[Proof of Lemma \ref{lem:covariance_residuals}]
 The first independence property is a classical consequence of the
 conditional regression representation for Gaussian random vectors,
 see, for example, \cite{MR1419991}. Since $\Var(\epsilon_i)$ is the conditional variance of $Y_i$ given $Y^{(-i)}$, it equals $[(\bGamma^{-1}(\phi))_{i,i}]^{-1}=\sigma_{\phi}^2$. Furthermore, 
\[\Cov(\epsilon_i,Y_i) = \Var(\epsilon_i)+ \sum_{v\in \bbN_h}\phi_j\Cov(\epsilon_i,Y_{i+v})= \Var(\epsilon_i)\ ,\]
by the independence of $\epsilon_i$ and $Y^{(-i)}$. Finally, consider any $i\neq j$, 
\beqn
\Cov(\epsilon_i,\epsilon_j)= \Cov(\epsilon_i,Y_j) - \sum_{v\in \bbN_h}\phi_v\Cov(\epsilon_i,Y_{j+v})\ ,
\eeqn 
where all the terms are equal to zero with the possible exception of $v=i-j$. The result follows.
\end{proof}

\begin{lem}[Comparison of $\bGamma^{-1}(\phi)$ and  $\bGamma_S^{-1}(\phi)$]
\label{lem:control_Gamma_S}
As long as $\|\phi\|_1<1$, the following properties hold:
\begin{enumerate}
 \item If $i\in S^h$ or if $j \in S^h$, then $(\bGamma_S^{-1}(\phi))_{i,j}= (\bGamma^{-1}(\phi))_{i,j}$.
\item If $i\in S^h$ and $j\in \Delta_h(S)$, then $1\leq (\bGamma_S^{-1}(\phi))_{j,j}\leq (\bGamma_S^{-1}(\phi))_{i,i}$.
\item If $i\in \Delta_h(S)$, then $\sum_{j\in S:j\neq i}(\bGamma_S^{-1}(\phi))_{i,j}^2\leq  \frac{2\|\phi\|_2^2}{(1-\|\phi\|_1)^3}$.
\end{enumerate}
\end{lem}

\begin{proof}
We prove each part in turn.

{\em Part 1.}
Consider $i \in S^h$ and any $j \in S$. By the Markov property,
conditionally to $(Y_{i+k},\ 1\leq |k|_{\infty}\leq h)$, $Y_i$ is
independent of all the remaining variables. Since all vertices $i+k$
with $1\leq |k|_{\infty}\leq h$ belong to $S$,  the conditional
distribution of $Y_{i}$ given $Y^{(-i)}$ is the same as the conditional
distribution of $Y_{i}$ given $(Y_{j}, j\in S\setminus\{i\})$. This
conditional distribution characterizes the $i$-th row of the inverse covariance matrix
$\bGamma^{-1}_S$. Also,
the conditional variance of $Y_i$
given $Y^{(-i)}$ is $[(\bGamma^{-1}(\phi))_{i,i}]^{-1}$ and the conditional variance of $Y_i$ given $Y_{S}$ is $[(\bGamma_S^{-1}(\phi))_{i,i}]^{-1}$. Furthermore, $-(\bGamma^{-1}(\phi))_{i,j}/(\bGamma^{-1}(\phi))_{i,i}$ is the $j$-th parameter of the condition regression of $Y_i$ given $Y^{(i)}$,
and therefore we conclude that
 $(\bGamma^{-1}(\phi))_{i,i}= (\sigma^2_{\phi})^{-1}=(\bGamma_S^{-1}(\phi))_{i,i} $ and $(\bGamma^{-1}(\phi))_{i,j}/(\bGamma^{-1}(\phi))_{i,i}= -\phi_{i-j} = (\bGamma_S^{-1}(\phi))_{i,j}/(\bGamma_S^{-1}(\phi))_{i,i}$.

{\em Part 2.}
Consider any vertex $i\in S^h$ and $j\in \Delta_h(S)$. Since
$1/\bGamma_S^{-1}(\phi))_{j,j}$ and $1/\bGamma_S^{-1}(\phi))_{j,j}$
are the conditional variances of $Y_i$ and $Y_j$ given $Y_{k}, k \in
S\setminus \{j\}$ and $Y_{k}, k \in S\setminus \{i\}$, respectively, we have
\beqn
1/\bGamma_S^{-1}(\phi))_{j,j}
&=& \Var\big(Y_j \big| Y_{k}: k \in S\setminus \{j\}\big)\\ 
&\geq& \Var\big(Y_j
\big|Y^{(-j)} \big) \\ &= & \Var\big(Y_i
\big|Y^{(-i)} \big)\quad\quad \quad  \text{(by stationarity of $Y$)}\\ & = & \Var\big(Y_i
\big|Y_k: k\in S\setminus\{i\}\big)\quad \quad \text{(since the neighborhood of $i$ is included in S)}\\
&= &
1/\bGamma_S^{-1}(\phi))_{i,i}~ .
\eeqn

{\em Part 3.}
Consider $i\in \Delta_h(S)$. The vector
$(\bGamma_S^{-1}(\phi))_{i,-i}\defeq \left(-(\bGamma_S^{-1}(\phi))_{i,j}/
  (\bGamma_S^{-1}(\phi))_{i,i}\right)_{j\in S: j\neq i}$ is formed by
the regression coefficients of $Y_{i}$ on $(Y_{j}, j\in
S\setminus\{i\})$. Since the conditional variance of $Y_{i}$ given
$(Y_{j}, j\in S\setminus\{i\})$ is at least $\sigma_{\phi}^2$ (by
Parts 1 and 2), we get
%\gab{The equality in the second line below needs an explanation.}
\beqn
 1 - \sigma_{\phi}^2  &\geq& 1 - \Var\left(Y_i|Y_j: \ j\in
   S\setminus\{i\}\right)\\ 
&= &\Var\Big(\E\left\{Y_i|Y_j,\ j\in S\setminus\{i\}\right\}\Big) \\
& = & \Var\left(\sum_{j\in S\setminus\{i\}} -\frac{(\bGamma_S^{-1}(\phi))_{i,j}}{ (\bGamma_S^{-1}(\phi))_{i,i}} Y_j \right) \\
&\geq & \sigma_\phi^4 \Var\Big(\sum_{j\in S\setminus\{i\}}
-(\bGamma_S^{-1}(\phi))_{i,j} Y_j\Big) \\
& = & 
\sigma_\phi^4 (\bGamma_S^{-1}(\phi))_{i,-i}^\top \bGamma_S(\phi) (\bGamma_S^{-1}(\phi))_{i,-i} \\
&\geq &  \frac{\sigma_\phi^6}{1+\|\phi\|_1} \|(\bGamma_S^{-1}(\phi))_{i,-i}\|_2^2\ , 
\eeqn
where the equality in the second line above we use 
$\Var(Y_j)= 1$ and the law of total variance 
(i.e., $\Var(Y)= \EXP[Var(Y|\cB)]+ Var(E[Y|\cB])$)
and in the last line we use that the smallest eigenvalue of $\bGamma(\phi)$ (and also of $\bGamma_S(\phi)$)  is larger than $\sigma_{\phi}^2/(1+\|\phi\|_1)$ (Lemma \ref{lem:spectrum_gamma}). Rearranging this inequality and using the fact that $\|\phi\|_1 < 1$, we arrive at
\beqn 
\|(\bGamma_S^{-1}(\phi))_{i,-i}\|_2^2
&\le& \frac{1 - \sigma_\phi^2}{\sigma_\phi^6} (1 + \|\phi\|_1)\leq 2 \frac{1 - \sigma_\phi^2}{\sigma_\phi^6}
\\ &\leq&  \frac{2\|\phi\|_2^2}{\sigma_\phi^4(1-\|\phi\|_1)}\quad \quad     \text{(by \eqref{eq:lb_sigma})}\\
&\leq&  \frac{2\|\phi\|_2^2}{(1-\|\phi\|_1)^3} \quad \quad \text{(using Lemma \ref{lem:spectrum_gamma}).}
\eeqn
\end{proof}

\begin{lem}\label{lem:det}
For any  $\phi_1, \phi_2\in \Phi_h$, define 
\[B_{\phi_1,\phi_2}:=\left(\frac{\det(\bGamma_{S}^{-1}(\phi_1))\det(\bGamma_{S}^{-1}(\phi_2))}{\det(\bGamma_{S}^{-1}(\phi_1)+
    \bGamma_{S}^{-1}(\phi_2) -\bI_S)}\right)^{1/2}\ .
\]
(Note that $V_S$ defined in Proposition \ref{prp:non_overlap_parametric} equals the expected value of
$B_{\phi_1,\phi_2}$ 
when $\phi_1$ and $\phi_2$ are drawn independently from the
distribution $\pi$.)
Assuming that  $\|\phi_1\|_1\vee \|\phi_2\|_1<1/5$,  we have
\[\log B_{\phi_1,\phi_2} \leq \frac{1}{2}|S|\langle \phi_1,\phi_2\rangle+ 8 Q_S\ , \]
where
\beqn  
Q_S &:=& |S|\sum_{s_1,s_2,s_3=1}^2\big|\sum_{j,k\in \bbN_h}\phi_{s_1,j}\phi_{s_2,k}\phi_{s_3,k-j}\big|+ 15|S| (\|\phi_1\|_2^3\vee \|\phi_2\|_2^3) + |\Delta_{h}(S)| (\|\phi_1\|_2^2\vee \|\phi_2\|_2^2) \\ 
&&+~ 28 |\Delta_{2h}(S)|\left(|\Delta_{2h}(S)|\vee (|\bbN_h|+1)\right)^{1/2} (\|\phi_1\|_2^3\vee \|\phi_2\|_2^3) \ .
\eeqn
\end{lem}

\begin{proof}%[Proof of Lemma \ref{lem:det}]
Since for any $\phi$, the spectrum of $\bGamma_S^{-1}(\phi)$ lies between the extrema of the spectrum of $\bGamma^{-1}(\phi)$, by Lemma \ref{lem:spectrum_gamma}, we have  
\beqn
\frac{1-\|\phi\|_1}{\sigma^2_{\phi}}-1 \leq
\lambda^{\min}\left(\bGamma_S^{-1}(\phi)-\bI_S\right)
\leq \lambda^{\rm max}\left(\bGamma_S^{-1}(\phi)-\bI_S\right)
\leq \frac{1+\|\phi\|_1}{\sigma^2_{\phi}}-1\ ,
\eeqn
where $\lambda^{\min}(\bA)$ and $\lambda^{\max}(\bA)$ denote the smallest
and largest eigenvalues of a matrix $\bA$.
Since $\sigma^2_{\phi}\leq \var{Y_i}=1$, the left-hand side is larger than  $-\|\phi\|_1$, while relying on \eqref{eq:lb_sigma}, we derive  
\[\frac{1+\|\phi\|_1}{\sigma^2_{\phi}}-1\leq (\|\phi\|_1+1)\left[1+\frac{\|\phi\|_2^2}{1-\|\phi\|_1}\right]-1 \leq \frac{2\|\phi\|_1}{1-\|\phi\|_1}\ . \]
Consequently, as long as $\|\phi\|_1< 1/5$, the spectrum of $\bGamma_S^{-1}(\phi)$ lies in $(\tfrac45,\tfrac32)$.
This allows us to use the Taylor series of the logarithm, which for a
matrix $\bA$ with spectrum in $(\tfrac12, 2)$, gives
\beqn
\left|\log\big(\det(\bA)\big) - \tr\big[\bA-\bI_S\big]+\frac{1}{2}\tr\left[\big(\bA-\bI_S\big)^2\right]\right| \leq \frac{8}{3}\left|\tr\left[\big(\bA-\bI_S\big)^3\right]\right|\ . 
\eeqn
Applying this expansion to $\bGamma_{S}^{-1}(\phi_1)$, $\bGamma_{S}^{-1}(\phi_2)$ and $\bGamma_{S}^{-1}(\phi_1)+ \bGamma_{S}^{-1}(\phi_2)-\bI_S$, 

\beqn
2\log B_{\phi_1,\phi_2} &\leq&
V_1+ \frac{16}{3} V_2+ 8V_3+8V_4\ , \\
V_1&:=& \tr\Big[\big(\bGamma_{S}^{-1}(\phi_1)-\bI_S\big)\big(\bGamma_{S}^{-1}(\phi_2)-\bI_S\big)\Big]\ , \\
V_2&:= &  \Big|\tr\Big[\big(\bGamma_{S}^{-1}(\phi_1)-\bI_S\big)^3\Big]\Big|+\Big|\tr\Big[\big(\bGamma_{S}^{-1}(\phi_2)-\bI_S\big)^3\Big]\Big| \ , \\
V_3&:= & \Big|\tr\Big[\big(\bGamma_{S}^{-1}(\phi_1)-\bI_S\big)\big(\bGamma_{S}^{-1}(\phi_2)-\bI_S\big)\big(\bGamma_{S}^{-1}(\phi_1)-\bI_S\big)\Big]\Big| \ , \\
V_4 &:= & \Big|\tr\Big[\big(\bGamma_{S}^{-1}(\phi_2)-\bI_S\big)\big(\bGamma_{S}^{-1}(\phi_1)-\bI_S\big)\big(\bGamma_{S}^{-1}(\phi_2)-\bI_S\big)\Big]\Big|
\ .
\eeqn

\paragraph{Control of $V_1$.}
We use the fact that
\[
\tr\Big[\big(\bGamma_{S}^{-1}(\phi_1)-\bI_S\big)\big(\bGamma_{S}^{-1}(\phi_2)-\bI_S\big)\Big]
= 
\sum_{i,j\in S}\big((\bGamma_{S}^{-1}(\phi_1))_{i,j}-\delta_{i,j}\big)\big((\bGamma_{S}^{-1}(\phi_2))_{i,j}-\delta_{i,j}\big)~.
\]
To bound the right-hand side, first consider any node $i \in S^h$ in the $h$-interior of $S$. By the first part of Lemma \ref{lem:control_Gamma_S}, the $i$-th row of $\bGamma_S^{-1}(\phi)$ equals the restriction to $S$ of the $i$-th row of $\bGamma^{-1}(\phi)$.  Using the definition of $\phi_1,\phi_2$, we therefore have
\begin{eqnarray}
\label{v1-bound1} \lefteqn{\sum_{j\in S}\big((\bGamma_{S}^{-1}(\phi_1))_{i,j}-\delta_{i,j}\big)\big((\bGamma_{S}^{-1}(\phi_2))_{i,j}-\delta_{i,j}\big)} &&\\ 
\notag &=& (\bGamma^{-1}(\phi_1))_{i,i} (\bGamma^{-1}(\phi_2))_{i,i} \langle \phi_1,\phi_2\rangle + ((\bGamma^{-1}(\phi_1))_{i,i}-1)((\bGamma^{-1}(\phi_2))_{i,i}-1) \\
\notag &= & \frac{\langle \phi_1,\phi_2\rangle }{\sigma_{\phi_1}^{2}\sigma_{\phi_2}^{2}}+ \frac{(1-\sigma_{\phi_1}^{2})(1-\sigma_{\phi_2}^{2})}{\sigma_{\phi_1}^{2}\sigma_{\phi_2}^{2}} \\
\notag &= & \langle \phi_1,\phi_2\rangle+ \langle \phi_1,\phi_2\rangle \frac{1- \sigma_{\phi_1}^{2}\sigma_{\phi_2}^{2}}{\sigma_{\phi_1}^{2}\sigma_{\phi_2}^{2}}+\frac{(1-\sigma_{\phi_1}^{2})(1-\sigma_{\phi_2}^{2})}{\sigma_{\phi_1}^{2}\sigma_{\phi_2}^{2}} \\
\notag &\leq &\langle \phi_1,\phi_2\rangle+ \frac{3}{2}\frac{\|\phi_1\|_2^4+ \|\phi_2\|_2^4}{(1-\|\phi_1\|_1)(1-|\phi_2\|_1)}\ ,
\end{eqnarray}
using \lemref{spectrum_gamma} in the last line. 
Next, consider a node $i \in \Delta_h(S)$, near the boundary of $S$. Relying on Lemmas \ref{lem:spectrum_gamma} and \ref{lem:control_Gamma_S}, we get
\begin{eqnarray}
\sum_{j\in S}\big((\bGamma_{S}^{-1}(\phi_1))_{i,j}-\delta_{i,j}\big)^2&\leq& \frac{2\|\phi_1\|_2^2}{(1-\|\phi_1\|_1)^3}+ \big(1/\sigma_{\phi_1}^{2}-1\big)^2\nonumber\\
&\leq & \frac{2\|\phi_1\|_2^2}{(1-\|\phi_1\|_1)^3}+ \frac{\|\phi_1\|_2^4}{(1-\|\phi_1\|_1)^2}\leq \frac{3\|\phi_1\|_2^2}{(1-\|\phi_1\|_1)^3} \ ,\label{eq:upper_delta_h} 
\end{eqnarray}
since we assume that $\|\phi\|_1<1$. By the Cauchy-Schwarz inequality,
\beq\label{v1-bound2}
\sum_{j\in S}\big((\bGamma_{S}^{-1}(\phi_1))_{i,j}-\delta_{i,j}\big)\big((\bGamma_{S}^{-1}(\phi_2))_{i,j}-\delta_{i,j}\big)\leq
3\frac{\|\phi_1\|_2^2\vee \|\phi_2\|_2^2}{(1-\|\phi_1\|_1\vee \|\phi_2\|_1)^3}\ .
\eeq
Summing \eqref{v1-bound1} over $i\in S^h$ and \eqref{v1-bound2} over $i \in \Delta_h(S)$, we get
\[
V_1 \le  |S|\langle \phi_1,\phi_2\rangle+ \frac32 |S|\frac{\|\phi_1\|_2^4+ \|\phi_2\|_2^4}{(1-\|\phi_1\|_1)(1-\|\phi_2\|_1)}+ 3|\Delta_h(S)|\frac{\|\phi_1\|_2^2\vee \|\phi_2\|_2^2}{(1-\|\phi_1\|_1\vee \|\phi_2\|_1)^3}\ . 
\]

\paragraph{Control of $V_2$.}
We proceed similarly as in the previous step.
Note that
\begin{eqnarray*}
\tr\Big[\big(\bGamma_{S}^{-1}(\phi_1)-\bI_S\big)^3\Big]
& = & 
\sum_{i,j,k\in S}\big((\bGamma_{S}^{-1}(\phi_1))_{i,j}-\delta_{i,j}\big)\big((\bGamma_{S}^{-1}(\phi_1))_{j,k}-\delta_{j,k}\big)\big((\bGamma_{S}^{-1}(\phi_1))_{k,i}-\delta_{k,i}\big) \\
&\le &
\sum_{i\in S}\left|\sum_{j,k\in S}\big((\bGamma_{S}^{-1}(\phi_1))_{i,j}-\delta_{i,j}\big)\big((\bGamma_{S}^{-1}(\phi_1))_{j,k}-\delta_{j,k}\big)\big((\bGamma_{S}^{-1}(\phi_1))_{k,i}-\delta_{k,i}\big)\right|
~.
\end{eqnarray*}
First, consider a node $i$ in $S\setminus \Delta_{2h}(S)$. Here, we use $\Delta_{2h}(S)$ instead of $\Delta_{h}(S)$ so that we may replace $\bGamma_S^{-1}(\phi)$ below with $\bGamma^{-1}(\phi)$.
We use again Lemma \ref{lem:control_Gamma_S} to replace $(\bGamma_S^{-1}(\phi))_{i,j}$ by $(\bGamma^{-1}(\phi))_{i,j}$ in the sum 
\beqn
\lefteqn{\Big|\sum_{j\in S}\sum_{k\in S}\big((\bGamma_{S}^{-1}(\phi_1))_{i,j}-\delta_{i,j}\big)\big((\bGamma_{S}^{-1}(\phi_1))_{j,k}-\delta_{j,k}\big)\big((\bGamma_{S}^{-1}(\phi_1))_{k,i}-\delta_{k,i}\big)\Big|}\\
&\leq & \Big|\sum_{j,k\in \bbN_h}\frac{-\phi_{1,j}\phi_{1,k}\phi_{1,k-j}}{\sigma_{\phi_1}^6}\Big| +  3 \sum_{j\in \bbN_h}|\phi_{1,j}|^2\frac{1-\sigma^2_{\phi_1}}{\sigma^6_{\phi_1}}+ \frac{\big(1-\sigma^2_{\phi_1}\big)^3}{\sigma^6_{\phi_1}}\\
&\leq & \Big|\sum_{j,k\in \bbN_h}\frac{-\phi_{1,j}\phi_{1,k}\phi_{1,k-j}}{(1-\|\phi_1\|_1)^3}\Big|  + 4\frac{\|\phi_1\|_2^3}{(1-\|\phi_1\|_1)^3 }\ ,
\eeqn
using \lemref{conditional_variance} in the last line. 
Next, consider a node $i\in \Delta_{2h}(S)$. If $i\notin\Delta_h(S)$, then the support of $(\bGamma_{S}^{-1}(\phi_1))_{i,-i}$ is of size $|\bbN_h|$. If $i\in \Delta_{h}(S)$, then $\Delta_{2h}(S)\setminus\{i\}$ separates $\{i\}$ from $S\setminus \Delta_{2h}(S)$ in the dependency graph and the Global Markov property \citep{MR1419991} entails that \[Y_i\indep (Y_{k},\ k\in S\setminus \Delta_{2h}(S))|(Y_{k},\  k\in \Delta_{2h}(S)\setminus\{i\})\ ,\]
and  therefore the support of $(\bGamma_{S}^{-1}(\phi_1))_{i,-i}$ is of  size smaller than $|\Delta_{2h}(S)|$.
Using the Cauchy-Schwarz inequality and \eqref{eq:upper_delta_h},  we get
\beqn
\lefteqn{\sum_{j\in S}\sum_{k\in S}\big((\bGamma_{S}^{-1}(\phi_1))_{i,j}-\delta_{i,j}\big)\big((\bGamma_{S}^{-1}(\phi_1))_{j,k}-\delta_{j,k}\big)\big((\bGamma_{S}^{-1}(\phi_1))_{k,i}-\delta_{k,i}\big)}&&\\
&\leq& \sum_{j\in S}\big|(\bGamma_{S}^{-1}(\phi_1))_{i,j}-\delta_{i,j}\big| \big\|(\bGamma_{S}^{-1}(\phi_1))_{i,.}-\delta_{i,.}\big\|_2 \big\|(\bGamma_{S}^{-1}(\phi_1))_{j,.}-\delta_{j,.}\big\|_2\\
&\leq & \sqrt{|\Delta_{2h}(S)|\vee (|\bbN_h|+1)}\big\|(\bGamma_{S}^{-1}(\phi_1))_{i,.}-\delta_{i,.}\big\|_2\big\|(\bGamma_{S}^{-1}(\phi_1))_{i,.}-\delta_{i,.}\big\|_2 \sup_{j\in S}\big\|(\bGamma_{S}^{-1}(\phi_1))_{j,.}-\delta_{j,.}\big\|_2\\
&\leq & \sqrt{|\Delta_{2h}(S)|\vee (|\bbN_h|+1)}3^{3/2}\frac{\|\phi_1\|_2^3 }{(1-\|\phi_1\|_1)^{9/2}}\ .
\eeqn
In conclusion, 
\beqn
V_2&\leq& |S|\frac{\big|\sum_{j,k\in
    \bbN_h}\phi_{1,j}\phi_{1,k}\phi_{1,k-j}\big|+\big|\sum_{j,k\in
    \bbN_h}\phi_{2,j}\phi_{2,k}\phi_{2,k-j}\big|}{(1-\|\phi_1\|_1\vee
  \|\phi_2\|_1)^3} + 8|S| \frac{\|\phi_1\|_2^3\vee
  \|\phi_2\|_2^3}{(1-\|\phi_1\|_1\vee \|\phi_2\|_1)^3 }\\
& & + 11 |\Delta_{2h}(S)|(|\Delta_{2h}(S)|\vee (|\bbN_h|+1))^{1/2}\frac{ \|\phi_1\|_2^3\vee \|\phi_2\|_2^3 }{(1-\|\phi_1\|_1\vee \|\phi_2\|_1)^{9/2}}\ .
\eeqn
\paragraph{Control of $V_3+V_4$.}
Arguing as above, we obtain
\beqn
V_3+V_4&\leq& |S|\frac{\big|\sum_{j,k\in \bbN_h}\phi_{1,j}\phi_{1,k}\phi_{2,k-j}\big|+\big|\sum_{j,k\in \bbN_h}\phi_{1,j}\phi_{1,k}\phi_{2,k-j}\big|}{(1-\|\phi_1\|_1\vee \|\phi_2\|_1)^3} + 8|S| \frac{\|\phi_1\|_2^3\vee \|\phi_2\|_2^3}{(1-\|\phi_1\|_1\vee \|\phi_2\|_1)^3 }\\&+& 11 |\Delta_{2h}(S)|(|\Delta_{2h}(S)|\vee (|\bbN_h|+1))^{1/2}\frac{ \|\phi_1\|_2^3\vee \|\phi_2\|_2^3 }{(1-\|\phi_1\|_1\vee \|\phi_2\|_1)^{9/2}}\ .
\eeqn
\end{proof}

\subsection{Proof of Corollary \ref{cor:non_overlap_GMRF}}
%\label{sec:proof-non_overlap_GMRF}

As stated in Lemma \ref{lem:spectrum_gamma}, all eigenvalues of the
covariance operator
$\bGamma^{-1}(\phi)$ lie in
$(1-\|\phi\|_1,\tfrac{1+\|\phi\|_1}{1-\|\phi\|_1} )$. 
Since the spectrum of $\bGamma^{-1}_{S}(\phi)$ lies between the
extrema of the spectrum of $\bGamma^{-1}(\phi)$, 
and using the assumption that $\|\phi\|_1<1/2$, this entails
\beq \label{gamma-id-op}
\|\bGamma_S(\phi)-\bI_S\|\leq \max\left[\frac{2\|\phi\|_1}{1+\|\phi\|_1}, \frac{\|\phi\|_1}{1-\|\phi\|_1}\right]<1\ ,
\eeq
We now apply Proposition \ref{prp:non_overlap_parametric} with the
probability measure $\pi$ concentrating on $\phi$. In this case,
\[V_S = 
\frac{\det(\bGamma_{S}^{-1}(\phi))}{ \det(2 \bGamma_{S}^{-1}(\phi) -
  \bI_S)^{1/2}} = \det(\bI_S - (\bI_S - \bGamma_S(\phi))^2)^{-1/2},\]
and we get  
\beqn
\bar{R}^*_{\nu, \phi} &\ge& 1 - \frac{1}{2|\cC| } \left[\sum_{S\in \cC} \det(\bI_S - (\bI_S - \bGamma_S(\phi))^2)^{-1/2}\right]^{1/2}\\
&\ge & 1 - \frac{1}{2|\cC| } \left[\sum_{S\in \cC} \exp\left(\frac{\|\bGamma_S(\phi)-\bI_S\|_F^2}{2(1- \|\bGamma_S(\phi)-\bI_S\|)}\right) \right]^{1/2}\\
&\ge& 1 - \frac{1}{2|\cC| } \left[\sum_{S\in \cC} \exp\left(\frac{\|\bGamma_S(\phi)-\bI_S\|_F^2}{2(1- 2\|\phi\|_1)}\right)  \right]^{1/2}\ ,
\eeqn
where $\|\cdot\|_F$ denotes the Frobenius norm.
The second inequality above is obtained by applying the inequality
$1/(1-\lambda)\leq e^{\lambda/(1-\lambda)}$ for $0\leq \lambda<1$ to
the eigenvalues of $(\bGamma_S(\phi)-\bI_S)^2$, while the third
inequality follows from \eqref{gamma-id-op} and the fact that $\|\phi\|_1 < 1/2$. 
It remains to bound $\|\bGamma_S(\phi)-\bI_S\|_F^2$:
\beqn
\|\bGamma_S(\phi) -\bI_S\|_F^2&=& \sum_{(i, j \in  S), \ i\neq j}\Cor^2(Y_i,Y_j)\\ &\leq&  |S| \sum_{v\neq 0}\gamma_v^2\\
&\leq & 20|S| \|\phi\|_2^2\ ,
\eeqn
where we used Lemma \ref{lem:conditional_variance},
$\sigma_{\phi}^2\le 1$, and $\|\phi\|_2\leq \|\phi\|_1\leq 1/2$ in the last line.

\subsection{Proof of Theorem \ref{thrm3:non_overlap_GMRF}}
%\label{sec:proof-non_overlap_GMRF}

Recall the definition of the prior $\pi$ defined just before the statement of the theorem.
Taking the numerical constant
$C$ in \eqref{eq:condition_r2_lower} sufficiently small and relying on condition \eqref{eq:condition_Nh}, we have $\|\phi\|_1= r \sqrt{\bbN_h}<1/5$. Consequently, the support of $\pi$ is a subset of the parameter space $\Phi_h$ and we are in position to invoke Lemma \ref{lem:det}. 

Let $\phi_1,\phi_2$ be drawn independently according to the
distribution $\pi$ and denote by $\xi_1$ and $\xi_2$ the corresponding
random vectors defined on $\bbN'_h$. By Lemma \ref{lem:det}, 
\[ \log B_{\phi_1,\phi_2}
\leq |S|r^2\bbN_h^{-1} \langle \xi_1,\xi_2 \rangle + 8 Q_S\ ,
\]
where
\[
Q_S \leq 23|S| r^{3}\sqrt{|\bbN_h|} + |\Delta_{h}(S)|r^2+ 28
|\Delta_{2h}(S)|(|\Delta_{2h}(S)|\vee (|\bbN_h|+1))^{1/2} r^3\ .
\]
Since  $\langle \xi_1, \xi_2\rangle $ is distributed as the sum of  $|\bbN_h|/2$ independent Rademacher random variables, we deduce that 
\beqn
V_S&\leq& \cosh\left(\frac{r^2|S|}{|\bbN_h|}\right)^{|\bbN_h|/2}
\exp\left( 383(|S|\sqrt{|\bbN_h|}\vee |\Delta_{2h}(S)|^{3/2}) r^{3} +
  8|\Delta_{h}(S)|r^2\right) \\ 
&\leq & \exp\left(\frac{r^4|S|^2}{4|\bbN_h|} \bigwedge \frac{r^2|S|}{2}+  383(|S|\sqrt{|\bbN_h|}\vee|\Delta_{2h}(S)|^{3/2}) + 8|\Delta_{h}(S)|r^2\right)\ , 
\eeqn
since $\cosh(x)\leq \exp(x)\wedge \exp(x^2/2)$ for any $x>0$. Combining this bound with Proposition \ref{prp:non_overlap_parametric}, we conclude that the Bayes risk $\bar{R}^*_{\nu,\pi}$ is bounded from below by 
\beq\label{eq:lower_gmrf_complex}
1 - \frac{1}{2\sqrt{|\cC|}}\max_{S\in \cC}\exp\left(\frac{|S|^2r^4}{4|\bbN_h|} \bigwedge \frac{|S|r^2}{2}+  383\left(|S|\sqrt{|\bbN_h|+1}\vee|\Delta_{2h}(S)|^{3/2}\right) r^{3} + 8|\Delta_{h}(S)|r^2\right)\ .
\eeq
If the numerical constant $C$ in Condition
\eqref{eq:condition_r2_lower} is sufficiently small, then $\tfrac{|S|^2r^4}{4|\bbN_h|} \bigwedge \tfrac{|S|r^2}{2}\leq 0.5\log(|\cC|/a)$. Also, choosing $C_0$ small enough in condition \eqref{eq:condition_r2_lower}, relying on condition \eqref{eq:condition_Nh} and on $|\bbN_h|\geq 1$, we also have 
\[383\left(|S|\sqrt{|\bbN_h|+1}\vee|\Delta_{2h}(S)|^{3/2}\right) r^{3} + 8|\Delta_{h}(S)|r^2\leq 0.5 \log(|\cC|/a) \ .\]
Thus, we conclude that $\bar{R}^*_{\nu,\pi}\geq 1-a$.

\subsection{Proof of Corollary \ref{cor:lower_hypercube}}

We deduce the result by closely following the proof of Theorem \ref{thrm3:non_overlap_GMRF}. We first prove that $5r\sqrt{|\bbN_h|}\leq 1$ is satisfied for $n$ large enough. 
Starting from \eqref{eq:powerless_gmrf}, we have, for $n$ large enough,
\beqn
5r\sqrt{|\bbN_h|}&\leq& 5C_2^{1/2} \left(\frac{|\bbN_h|\log(\tfrac{n}{k})}{k}\bigvee\frac{|\bbN_{h}|^{3/2}\sqrt{\log(\tfrac nk)}}{k}\right)^{1/2}\\
&\leq& 5C_2^{1/2} \left(C_1 \bigvee\frac{|\bbN_{h}|^{3/2}\sqrt{\log(\tfrac nk)}}{k}\right)^{1/2}\ ,
\eeqn
where we used Condition \eqref{eq:condition_neigbhorhood} in the second line. Taking $C_1$ and $C_2$ small enough, we only have to bound $|\bbN_{h}|^{3/2}\sqrt{\log(\tfrac nk)}/k$. 
We distinguish two cases.
\bitem
\item {Case 1}: $|\bbN_h|\leq \log(n/k)$. Since $|\bbN_h|\leq  C_1k/\log\left(\frac{n}{k}\right)$, it follows that $|\bbN_{h}|^{3/2}\sqrt{\log(\tfrac nk)}/k\leq C_1$. 

\item {Case 2}: $|\bbN_h|\geq \log(n/k)$. Then the second part of Condition \eqref{eq:condition_neigbhorhood}  enforces $\log^{4/5}(n/k) \leq C_1 k^{2/5}$. Using again the second part of Condition \eqref{eq:condition_neigbhorhood} yields 
\beqn
\frac{|\bbN_{h}|^{3/2}\sqrt{\log(\tfrac nk)}}{k}\leq C_1^{3/2} \frac{\log^{4/5}(n/k)}{k^{2/5}}\leq C_1^{3/2}\ .
\eeqn
\eitem
\bigskip

As  $5r\sqrt{|\bbN_h|}\leq 1$, we can use the same prior $\pi$ as in the proof of Theorem \ref{thrm3:non_overlap_GMRF} and arrive at the same lower bound \eqref{eq:lower_gmrf_complex} on $R^*_\pi$. It remains to prove that this lower bound  goes to one, namely that 
\[\frac{2|S|^2r^4}{|\bbN_h|} \bigwedge (|S|r^2)+  765\left(|S|\sqrt{|\bbN_h|+1}\vee|\Delta_{2h}(S)|^{3/2}\right) r^{3} + 16|\Delta_{2h}(S)|r^2- \frac{1}{2}\log(n/k)\to -\infty\ , \]
where $S$ is a hypercube of size $k$.
Taking the constant $C_2$ small enough in \eqref{eq:powerless_gmrf} leads to $\frac{2k^2r^4}{|\bbN_h|} \bigwedge (kr^2)\leq \log(n/k)/4$ for $n$ large enough. 
\beqn
k r^{3}\sqrt{|\bbN_h|}\leq C_2^{3/2}\left[\frac{\log(n/k)^3|\bbN_h|}{k}\bigvee \frac{\log(n/k)^{3/2}|\bbN_h|^{5/2}}{k} \right]^{1/2}\leq  C_2^{3/2}\left(C_1^{1/2}\vee C_1^{5/4}\right)\log(n/k)\ ,
\eeqn
where we used again the second part of  Condition \eqref{eq:condition_neigbhorhood}. Taking $C_1$ and $C_2$ small enough ensures that $765k r^{3}\sqrt{|\bbN_h|+1}\leq \log(n/k)/8$ for $n$ large enough. Finally, it suffices to control $|\Delta_{2h}(S)|^{3/2}r^3$ since $|\Delta_{2h}(S)|r^2\leq |\Delta_{2h}(S)|^{3/2}r^3\vee 1$.  Observe that
\[|\Delta_{2h}(S)|= \ell^{d}- (\ell-4h)^d= \ell^d \left[1- (1-4h/\ell)^d\right]\leq 4\ell^ddh/\ell\leq  4d|\bbN_h|^{1/d}k^{\frac{d-1}{d}} .\]
It then follows from Condition \eqref{eq:powerless_gmrf} that
\beqn
(d|\bbN_h|^{1/d}k^{\frac{d-1}{d}})^{3/2}r^3 &\leq & C_2^{3/2}\left[\frac{d^{3/2}|\bbN_h|^{3/(2d)}}{k^{3/(2d)}}\log^{1/2}\left(\frac{n}{k}\right)\bigvee \frac{d^{3/2}|\bbN_h|^{3/(2d)+3/4}}{k^{3/(2d)}\log^{1/4}\left(\frac{n}{k}\right)} \right] \log\left(\frac{n}{k}\right)\\
&\leq & C_2^{3/2}\left[C_1^{3/(2d)}d^{3/2}\log^{-1/4}\left(\frac{n}{k}\right)\bigvee C_1^{\frac{6+3d}{4d}} \right] \log\left(\frac{n}{k}\right)
\eeqn 
where we used again \eqref{eq:condition_neigbhorhood} in the second line. Choosing $C_1$ and $C_2$ small enough concludes the proof.

\subsection{Proof of \prpref{glrt}}

We leave $\phi$ implicit throughout.  
Define 
\[
U'_S = X_S^\top (\bI_S - \bGamma_S^{-1}) X_S - \tr(\bI_S - \bGamma_S^{-1}).
\]
Under the null, $X$ is standard normal, so applying the union bound and \lemref{normal-quad} gives
\[
\pr{U > 4} \le \sum_{\cS \in \cC}\P\Big\{ U'_S > 4\|\bI_S - \bGamma_S^{-1}\|_F\sqrt{\log(|\cC|)} +  4\|\bI_S - \bGamma_S^{-1}\|\log(|\cC|)\Big\} \le |\cC|^{-1}~.
\]
%Therefore, with our choice $t = 4 \sqrt{\log N}$, and assuming the first part of \eqref{glrt1}, we get
%\[\pr{U > t} \le \frac1N~.\]

Under the alternative where $S \in \cC$ is anomalous, $X_S$ has covariance $\bGamma_S$, so that we have $X_S^\top (\bI_S - \bGamma_S^{-1}) X_S \sim Z^\top (\bGamma_S - \bI_S) Z$, where $Z$ is standard normal in dimension $|S|$. Since $\Var(Y_i)=1$, the diagonal elements of $\bGamma_S - \bI_S$ are all equal to zero. We apply \lemref{normal-quad} to get that
\beqn
\P\left[X_S^\top (\bI_S - \bGamma_S^{-1}) X_S\leq  - 2\|\bGamma_S - \bI_S\|_F\sqrt{\log(|\cC|)}-   2\|\bGamma_S - \bI_S\|\log(|\cC|) \right]\leq |\cC|^{-1}\ , 
\eeqn 
In view of the definition of $U$, we have  $\P[U>4]\geq 1-  |\cC|^{-1}$ as soon as\beq\label{eq:lower_RS}
\tr[\bGamma_S^{-1}-\bI_S] \ge 4 \big[\|\bGamma_S - \bI_S\|_F \vee \|\bGamma_S^{-1}-\bI_S\|_F\big] \sqrt{\log(|\cC|)} +  6 \big[\|\bGamma-\bI\|\vee \|\bGamma^{-1}-\bI\| \big]\log(|\cC|)\ .
\eeq
Therefore, it suffices to bound $\|\bGamma_S - \bI_S\|_F$, $\|\bGamma_S^{-1}-\bI_S\|_F$, $\|\bGamma-\bI\|$, $\|\bGamma^{-1}-\bI\|$ and $\tr[\bGamma_S^{-1}-\bI_S]$. 
In the sequel, the $C$ denotes a large enough positive constant depending only on $\eta$, whose value may vary from line to line.
From Lemma \ref{lem:conditional_variance}, we deduce that
\[
\|\bGamma_S - \bI_S\|^2_F\leq C |S| \|\phi\|_2^2 \ . 
\]
Lemma \ref{lem:spectrum_gamma} implies that 
\[\|\bGamma-\bI\|\vee \|\bGamma^{-1}-\bI\| \leq C \ . \]
We apply Lemma~\ref{lem:control_Gamma_S} to obtain
\beqn
\|\bGamma_S^{-1}-\bI_S\|_F^2 &\leq& C |S|\|\phi\|_2^2+ |S|(\sigma_{\phi}^{-2} - 1)^2\\
&\leq & C |S|\|\phi\|_2^2\ ,
\eeqn
where we used Lemma~\ref{lem:spectrum_gamma} in the second line.
Finally, we use again Lemmas~\ref{lem:control_Gamma_S} and~\ref{lem:spectrum_gamma} to obtain
\beqn
 \tr[\bGamma_S^{-1}-\bI_S]&=& |S^h| \frac{\sigma^2_{\phi}-1}{\sigma^2_{\phi}}+ \sum_{j\in \Delta_h(S)}(\bGamma^{-1}_{S})_{j,j} -1 \\
&\geq & |S^h| \frac{\sigma^2_{\phi}-1}{\sigma^2_{\phi}} \geq C |S|\|\phi\|^2_2\ .
\eeqn
Consequently, \eqref{eq:lower_RS} holds as soon as $|S|\|\phi\|_2^2\geq C \log(|\cC|)$.

\subsection{Proof of \thmref{LS1}}

We use $C, C', C''$ as generic positive constants, whose actual values may change with each appearance.

\medskip

\noindent
{\bf Under the null hypothesis}.
 First, we bound the $1-\alpha$ quantile of $T^*$ under the null hypothesis. 
Denote $Z_S:= \|\boldsymbol{\Pi}_{S,h}X_{S,h}\|_2^2$ so that $T_S= Z_S|S^h|[\|X_{S,h}\|_2^2-Z_S]^{-1}$. Since $Z_S$ is the squared norm of the projection of $X_{S,h}$ onto the column space of $\bF_{S,h}$, we can express $Z_S$ as a least-squares criterion: 
\[Z_S= \max_{\phi\in \mathbb{R}^{\bbN_h}} \|X_{S,h}\|_2^2- \sum_{i\in S^h}\Big(X_i - \sum_{j \in \bbN_h }\phi_jX_{i+j} \Big)^2 \ .\]
Given $\phi\in \mathbb{R}^{\bbN_h}$, define the matrix $\bB_{\phi,S}\in \mathbb{R}^{S\times S}$ such that  for any $i\in S^h$, and any $j$, $(\bB_{\phi,S})_{i,i+j}=\phi_j$, and all the remaining entries of $\bB_{\phi,S}$ are zero. It then follows that 
\beq\label{eq:definition_ZS}
Z_S=  \max_{\phi\in \mathbb{R}^{\bbN_h}} \tr\left[\bR_{\phi,S}X_{S}X_{S}^\top \right]\ , \quad \quad \bR_{\phi,S}:= (\bI_S-\bB^\top _{\phi,S})(\bI_S-\bB_{\phi,S})-\bI_S\ .
\eeq
Observe that  $Z_S$ can be seen as the supremum of a Gaussian chaos of order 2.
As the collection of matrices in the supremum of \eqref{eq:definition_ZS} is not bounded, we cannot directly apply Lemma \ref{lem:chaos}. Nevertheless, upon defining defining 
$\tilde{Z}_S:= \max_{\|\phi\|_1\leq 1} \tr\left[\bR_{\phi,S}X_{S}X_{S}^\top \right]$,
we have for any $t>0$, 
\beq\label{eq:deviation_ZS}
\P[Z_S\geq t]\leq \P[\tilde{Z}_S\geq t]+ \P[\tilde{Z}_S\neq Z_S]\ ,
\eeq and we can control the deviations of $\tilde{Z}_S$ using Lemma \ref{lem:chaos}. 
Observe that for any $\phi$ with $\|\phi\|_1\leq 1$, $\|\bI_S-\bB_{\phi,S}\| \le 2$, so that $\|\bR_{\phi,S}\|\leq 3$. 
Choose $\widehat{\phi}_S$ among the $\phi$'s achieving the maximum in \eqref{eq:definition_ZS}, and note that
$\P[\tilde{Z}_S\neq Z_S] = \P[\|\widehat{\phi}_S\|_1> 1]$.  We bound the right-hand side below.  
In view of \lemref{chaos}, we also need to bound $\E[\tilde{Z}_S]$ and $\E\big[\sup_{\|\phi\|_1\leq 1}\tr(\bR_{\phi,S} X_{S}X_{S}^\top \bR_{\phi,S})\big]$ in order to control $\P[\tilde{Z}_S\geq t]$.

%%%%%%%%%%%
\medskip\noindent
{\it Control of $\P[\|\widehat{\phi}_S\|_1> 1]$}. When $\bF_{S,h}^\top \bF_{S,h}$ is invertible, $\widehat{\phi}_S= (\bF_{S,h}^\top \bF_{S,h})^{-1}\bF_{S,h}X_{S,h}$. By the Cauchy-Schwarz inequality, 
\begin{eqnarray}
\P\big[\|\widehat{\phi}_S\|_1> 1\big]
&\leq & \P\big[\|\widehat{\phi}_S\|_2 > |\bbN_h|^{-1/2}\big]\nonumber\\
&\leq & \P\left[\lambda^{\min}(\bF_{S,h}^\top \bF_{S,h})\leq \tfrac12 |S^h|\right]+ \P\left[\|\bF_{S,h}X_{S,h}\|_2\geq \frac{|S^h|}{2|\bbN_h|^{1/2}} \right]\nonumber\\
&\leq &\P\left[\lambda^{\min}(\bF_{S,h}^\top \bF_{S,h})\leq \tfrac12 |S^h|\right]+ \P\left[\|\bF_{S,h}X_{S,h}\|_\infty\geq \frac{|S^h|}{2|\bbN_h|} \right]
\label{eq:upper_phi_1}\ .
\end{eqnarray}
First, we control the smallest eigenvalue of $\bF_{S,h}^\top \bF_{S,h}$. Under the null hypothesis, the vectors $F_{i}$ follow the standard normal distribution, but $\bF_{S,h}^\top \bF_{S,h}$ is not a Wishart matrix since the vectors $F_i$ are correlated. 
However, $\bF_{S,h}^\top \bF_{S,h}$ decomposes as a  sum of  $|\bbN_{h}|+1$ (possibly dependent) standard Wishart matrices. 
Indeed, define  
\beq \label{S_i}
S_i= S^h\cap \{i + (2h+1)u,\ u\in \mathbb{Z}^d\}, \quad i\in \bbN_h \cup \{0\}\ ,
\eeq
and then $\bA_i=\sum_{j\in S_i} F_jF_j^\top $. 
The vectors $(F_{j},\ j\in S_i)$ are independent since the minimum $\ell_{\infty}$ distance between any two nodes in $S_i$ is at least $2h+1$, so that $\bA_i$ is standard Wishart. 
Denoting $n_i=|S_i|$, we are in position to apply \lemref{concentration_vp_wishart}, to get
\[\P\left[\lambda^{\min}(\bA_i)\le n_i - 2\sqrt{|\bbN_h| n_i}- 2\sqrt{2x n_i}\right]\leq e^{-x}\ , \quad \forall x > 0\ .\]
Since the $\{S_i : i \in \bbN_h \cup \{0\}\}$ forms a partition of $S^h$, we have $\bF_{S,h}^\top \bF_{S,h}= \sum_{i \in \bbN_h \cup \{0\}} \bA_i$, and in particular, $\lambda_{\rm min}(\bF_{S,h}^\top \bF_{S,h}) \ge \sum_{i} \lambda_{\rm min}(\bA_i)$.
Using this, the tail bound for $\lambda^{\min}(\bA_i)$ with $x \gets x +\log(|\bbN_h|+1)$, some simplifying algebra, and the union bound, we conclude that, for all $x > 0$,
\beq\label{eq:control_eigen_FF}
\P\left[\lambda^{\min}(\bF_{S,h}^\top \bF_{S,h})\le |S^h|- 5(|\bbN_h|+1)\sqrt{|S^h|}- 3\sqrt{(|\bbN_h|+1)|S^h|x}\right]\leq e^{-x}\ ,
\eeq
since
\[
\sum_{i \in \bbN_h \cup \{0\}} \Big(n_i - 2\sqrt{|\bbN_h| n_i}- 2\sqrt{2x n_i}\Big) 
\ge \sum_{i \in \bbN_h \cup \{0\}} n_i - 2(\sqrt{|\bbN_h|} + \sqrt{2x}) \sum_{i \in \bbN_h \cup \{0\}} \sqrt{n_i}\ ,
\]
with $\sum_{i \in \bbN_h \cup \{0\}} n_i = |S^h|$, and $\sum_{i \in \bbN_h \cup \{0\}} \sqrt{n_i}\leq \sqrt{|S^h|(|\bbN_h|+1)}$, by the Cauchy-Schwarz inequality. 
Taking $x=C |S^h|/(|\bbN_h|+1)$ in the above inequality for a sufficiently small constant $C$ and relying on Condition \eqref{eq:condition_LS}, we get
\[
\P\left\{\lambda^{\min}(\bF_{S,h}^\top \bF_{S,h}^\top )\leq \tfrac12 |S^h|\right\}
\le \exp\left(-C |S^h|/(|\bbN_h|+1)\right) \ .
\]
We now turn to bounding $\|\bF_{S,h}X_{S,h}\|_{\infty}$.
Each component of $\bF_{S,h}X_{S,h}$ is of the form $Q_v:=\sum_{i\in S^h}X_{i}X_{i+v}$ for some $v\in \bbN_h$. 
Note that $Q_v$ is a quadratic form of $|S|$ standard normal variables, and the corresponding symmetric matrix has zero trace, Frobenius norm equal to  $\sqrt{|S^h|/2}$, and operator norm smaller than $1$ by diagonal dominance. 
Combining Lemma~\ref{lem:normal-quad} with a union bound, we get
\beqn
\P\left\{\|\bF_{S,h}X_{S,h}\|_{\infty}\geq \sqrt{2|S^h|(x+|\bbN_h|)}+ 2 (x+|\bbN_h|) \right\} \leq 2e^{-x}\ , \forall x>0\ .
\eeqn
Taking $x=C|S^h|/|\bbN_h|^2$ in the above inequality for a sufficiently small constant $C$ and using once again Condition \eqref{eq:condition_LS} allows us to get the bound 
\[\P\left\{ \|\bF_{S,h}X_{S,h}\|_\infty\geq \frac{|S^h|}{2|\bbN_h|} \right\} \le \exp[-C |S^h|/|\bbN_h|^2]\ .\]
Plugging these bounds into \eqref{eq:upper_phi_1}, we conclude that  
\beq\label{eq:bound_phi_l1}
\P\left\{\|\widehat{\phi}_S\|_1> 1\right\} 
\leq 3\exp\left(-C \frac{|S^h|}{|\bbN_h|^2}\right)\ .
\eeq

%%%%%%%%%%%
\medskip\noindent 
{\it Control of $\E[\tilde{Z}_S]$}.
Since 
\[\tilde{Z}_S\leq Z_S = \|\boldsymbol{\Pi}_{S,h}X_{S,h}\|_2^2\leq \|(\bF_{S,h}^\top \bF_{S,h})^{-1}\|\|\bF_{S,h}X_{S,h}\|_2^2 \le \|X_{S,h}\|_2^2\ ,\] 
we have, for any $a > 0$,
\beqn
\E[\tilde{Z}_S]
&\leq &  a \E\left[\|\bF_{S,h}X_{S,h}\|_2^2\right] + \E\left[\|X_{S,h}\|_2^2 \IND{\|(\bF_{S,h}^\top \bF_{S,h})^{-1}\|\geq a}\right] \\
&\leq &  a \E\left[\|\bF_{S,h}X_{S,h}\|_2^2\right] + 
\sqrt{\P\left\{\|(\bF_{S,h}^\top \bF_{S,h})^{-1}\|\geq a\right\} \E\left[ \|X_{S,h}\|_2^4\right]}\ ,
\eeqn
where we used the Cauchy-Schwarz inequality in the second line.
Since, under the null, $X_S\sim \cN(0,\bI_S)$, it follows that  $\E\left[\|\bF_{S,h}X_{S,h}\|_2^2\right]=|\bbN_h(S)||S^h|$ and $\E\left[\|X_{S,h}\|_2^4\right]= |S^h|(|S^h|+2)$.  Gathering this, the deviation inequality \eqref{eq:control_eigen_FF} with $x=C|S^h|/|\bbN_h|^2$ with a small constant $C>0$, and Condition \eqref{eq:condition_LS}, and choosing as threshold $a =(|S^h|(1-|\bbN_h|^{-1/2})^{-1}$, leads to
\begin{eqnarray}
 \E[\tilde{Z}_S]&\leq & \frac{|\bbN_h|}{1-|\bbN_h|^{-1/2}}+ \sqrt{3}|S^h|\sqrt{\P\left\{\lambda^{\min}(\bF_{S,h}^\top \bF_{S,h})\le 1/a\right\}}\nonumber\\
&\leq & |\bbN_h|+ C' |\bbN_h|^{1/2}+ \sqrt{3}|S^h| \exp\left(-C\frac{|S^h|}{|\bbN_h|^2}\right)\nonumber\\
&\leq & |\bbN_h|+ C |\bbN_h|^{1/2}\label{eq:upper_EZ}\ .
\end{eqnarray}

%%%%%%%%%%%
\medskip\noindent 
{\it Control of $\E\big[\sup_{\|\phi\|_1\leq 1} \tr(\bR_{\phi,S} X_{S}X_{S}^\top \bR_{\phi,S})\big]$}. As explained above, $\|\bR_{\phi,S}\|\leq 3$ and we are therefore  able to bound this expectation in terms of $\E[\tilde{Z}_S]$ as follows:
\beq\label{eq:upper_ER2}
\E\left[\sup_{\|\phi\|_1\leq 1} \tr(\bR_{\phi,S} X_{S}X_{S}^\top \bR_{\phi,S})\right]\leq 3\E\big[ \tilde{Z}_S\big]\leq C |\bbN_h|\ ,
\eeq
where we used \eqref{eq:upper_EZ} in the last inequality. 

%%%%%%%%%%%
\medskip
Combining the decomposition \eqref{eq:deviation_ZS} with Lemma \ref{lem:chaos} and \eqref{eq:bound_phi_l1}, \eqref{eq:upper_EZ} and \eqref{eq:upper_ER2}, we obtain
\[\P\left\{Z_S\geq |\bbN_h|+ C\left(|\bbN_h|^{1/2}+\sqrt{|\bbN_h|t}+ t\right)\right\}\leq e^{-t}+ 3\exp\left(-C' \frac{|S^h|}{|\bbN_h|^2}\right)\ , \quad \forall t>0\ . \]
Since 
\[T_S= \frac{|S^h|Z_S}{\|X_{S,h}\|_2^2-Z_S}, \text{ where $\|X_{S,h}\|_2^2$ follows a $\chi^2$ distribution with $|S^h|$ degrees of freedom}\ ,\] 
from Lemma \ref{lem:normal-quad}, we derive 
\[\P\Big[\|X_{S,h}\|_2^2\geq |S^h|-2 \sqrt{|S^h|t}-2t\Big]\leq e^{-t}\ ,\] 
for any $t>0$, and from these two deviation inequalities,  we get,
for all $t\leq C'' |S^h|$,
\beqn
\P\left[T_S \geq  \frac{|\bbN_h|+ C\big(|\bbN_h|^{1/2}+\sqrt{|\bbN_h|t}+ t\big)}{1- C \left[\sqrt{ \frac{t}{|S^h|}}\vee\frac{|\bbN_h|}{|S^h|} \right]} \right]\leq 2e^{-t}+  3\exp\left[-C' \frac{|S^h|}{|\bbN_h|^2}\right]  \ . 
\eeqn
Finally, we take a union bound over all $S\in \cC$ and invoke again Condition \eqref{eq:condition_LS} to conclude that, for any $t\leq C'' |S^h|$,
\[
\P\left\{\max_{S\in \cC} T_S \geq  |\bbN_h|+ C\left(\sqrt{|\bbN_h|(\log(|\cC|)+ 1+t )}+ \log(|\cC|)+ t\right) \right\}\leq 2e^{-t}+  3|\cC|\exp\left(-C' \frac{|S^h|}{|\bbN_h|^2}\right) \ .
\]
To conclude, we let $t=\log(1/(4\alpha))$ in the above inequality, and use the condition on $\alpha$ in the statement of the theorem together with Condition \eqref{eq:condition_LS}, to get the following control of $T^*$ under the null hypothesis:
\[
\P\left\{\max_{S\in \cC} T_S \geq   |\bbN_h|+ C\left(\sqrt{|\bbN_h|(\log(|\cC|)+ 1+\log(\alpha^{-1}) )}+ \log(|\cC|)+ \log(\alpha^{-1})\right) \right\} \leq \alpha \ .
\]

\medskip

\noindent
{\bf Under the alternative hypothesis}.
Next we study the behavior of the test statistic $T^*$ under the assumption
that there exists some $S\in \cC$ such that $X_S=Y_S\sim \cN(0, \bGamma_S(\phi))$. Since $T^*\geq T_S$, it suffices to focus on this particular $T_S$. For any $i \in S^h$, recall that $Y_i= \phi^\top F_i+ \epsilon_i$ where $F_i = (Y_{i+v} : 1 \le |v|_{\infty} \le h)$ and $\epsilon_i$ is independent of $F_i$. Hence, $Z_S$ decomposes as 
\beqn
Z_S &= &\|\boldsymbol{\Pi}_{S,h}Y_{S,h}\|_2^2 \\&= &\|\bF_{S,h}\phi + \boldsymbol{\Pi}_{S,h}\epsilon_{S,h} \|_2^2\\
&= &  \|\bF_{S,h}\phi\|_2^2+ 2 \phi^\top  \bF_{S,h}^\top \epsilon_{S,h}+ \|\boldsymbol{\Pi}_{S,h}\epsilon_{S,h}\|_2^2 = ({\rm I}) + ({\rm II}) +({\rm III})\ .
\eeqn
To bound the numerator of $T_S$, we bound each of these three terms.
(I) and (II) are simply quadratic functions of multivariate normal random vectors and we control their deviations using Lemma \ref{lem:normal-quad}. In contrast, (III) is more intricate and we use an ad-hoc method. 
In order to structure the proof, we state four lemmas needed in our calculations. We provide proofs of the lemmas further down.

\begin{lem}\label{lem:control_(I)}
Under condition \eqref{eq:condition_LS}, there exists a numerical constant $C>0$ such that 
\beq \label{eq:deviation_I}
\P\left\{({\rm I})\geq \frac{|S^h|\|\phi\|_2^2}{2(1+\|\phi\|_1)}\sigma_{\phi}^2  \right\} \ge  1 -\exp\left(-C\frac{|S^h|}{|\bbN_h|}\right)\ .
\eeq
\end{lem}

\begin{lem}\label{lem:control_(II)}
For any $t>0$, 
\begin{eqnarray}\label{eq:deviation_II-a}
\P\left\{ ({\rm II})\geq - 2\sigma_{\phi}\|\phi\|_2\sqrt{2 |S^h| \|\bGamma(\phi)\|(2+\|\phi\|_1)t}- 12\big[\|\bGamma(\phi)\|\vee (1+\|\phi\|_1)\sigma_{\phi}^2\big]t \right\} &\geq& 1 -e^{-t}\ ,\\ \label{eq:deviation_II-b}
\P\left\{({\rm II}) \geq   -2\sqrt{2} \sigma_{\phi}\sqrt{(|\bbN_h|+1)[\log(|\bbN_h|+1)+ t]} \|\bF_{S,h} \phi\|_2\right\}&\geq& 1-e^{-t}\ .
\end{eqnarray}
 \end{lem}

Recall that $\gamma_j= (\bGamma(\phi))_{0,j}$ denotes the covariance between $Y_0$ and $Y_j$.
\begin{lem}\label{lem:control_(III)} 
Denote by $\bGamma_{\bbN_h}(\phi)$ the covariance matrix of $(Y_{i},\ i\in \bbN_h)$.
For any $t\leq |S^h|$, 
\beq\label{eq:deviation_wisharts}
\lambda^{\rm max}\left(\bGamma_{\bbN_h}(\phi)^{-1/2}\frac{\bF_{S,h}^\top  \bF_{S,h}}{|S^h|}\bGamma_{\bbN_h}(\phi)^{-1/2}\right)\leq 1+
4\|\bGamma(\phi)\|\|\bGamma^{-1}(\phi)\|\frac{|\bbN_h|}{|S^h|^{1/2}}\left(\sqrt{t}+\log(|\bbN_h|)\right)
\eeq
with probability larger than $1-2e^{-t}$. Also, for any $t\ge 1$, 
\beq\label{eq:control_bSigma}
\frac{\|\bGamma_{\bbN_h}(\phi)^{-1/2}\bF_{S,h}^\top  \epsilon_{S,h}\|_2^2}{|S^h|\sigma_{\phi}^2}\geq |\bbN_h|- C\left(|\bbN_h|\|\phi\|_1+ \|\bGamma^{-1}(\phi)\|\|\phi\|_2^2+ \frac{|\bbN_h|^{5/2} + |\bbN_h|^2 (\sum_{j\neq 0}\gamma_{j}^2)}{\sqrt{|S^h|}} \right)t^2
\eeq
with probability larger than $1-2e^{-t}$.
\end{lem}

To bound the denominator of $T_S$, we start from the inequality
\[
\|Y_{S,h}\|_2^2 - \|\boldsymbol{\Pi}_{S,h}Y_{S,h}\|_2^2= \|\epsilon_{S,h}\|_2^2- \|\boldsymbol{\Pi}_{S,h}\epsilon_{S,h}\|_2^2 \leq \|\epsilon_{S,h}\|_2^2
\]
and then use the following result.

\begin{lem}\label{lem:control_(IV)}
Under condition \eqref{eq:condition_LS}, we have
\beq\label{eq:upper_epsilon}
 \P\left\{\|\epsilon_{S,h}\|^2 \leq \sigma_{\phi}^2|S^h|(1+ |\bbN_h|^{-1/2})\right\}\geq 1-\exp\left(-C \frac{|S^h|}{|\bbN_h|^2}\right) \ . 
\eeq
\end{lem}

With these lemmas in hand, we divide the analysis into two cases depending on the value of $\|\phi\|_2^2$. For small $\|\phi\|^2_2$, the operator norm of the covariance operator $\bGamma(\phi)$ remains bounded, which simplifies some deviation inequalities. For large $\|\phi\|^2_2$, we are only able to get looser bounds which are nevertheless sufficient as in that case $\|\phi\|_2^2$ is far above the detection threshold. 

%%%%%%%%%%%
\medskip\noindent 
{\bf Case 1}: $\|\phi\|_2^2\leq (4|\bbN_h|)^{-1}$.  This implies that $\|\phi\|_1\leq 1/2$ and also that $\|\bGamma(\phi)\|\leq 2\sigma_{\phi}^2$ by Lemma \ref{lem:spectrum_gamma}. Combining \eqref{eq:deviation_I} and \eqref{eq:deviation_II-a} together with the inequality 
$2xy\leq x^2+y^2$, we derive that for any $t>0$, 
\beq\label{eq:lower_case1}
\frac{({\rm I})+ ({\rm II})}{\sigma^2_{\phi}}\geq C\left(|S^h|\|\phi\|_2^2- t\right)
\eeq
with probability larger than $1-e^{-t} -\exp\left(-C\frac{|S^h|}{(|\bbN_h|+1)}\right)$. Turning to the third term, we have
\[\frac{({\rm III})}{\sigma_{\phi}^2} \geq \lambda^{\rm max}\left(\bGamma_{\bbN_h}(\phi)^{-1/2}\frac{\bF_{S,h}^\top  \bF_{S,h}}{|S^h|}\bGamma_{\bbN_h}(\phi)^{-1/2}\right)^{-1}\frac{\|\bGamma_{\bbN_h}(\phi)^{-1/2}\bF_{S,h}^\top \epsilon_{S,h}\|_2^2}{\sigma_{\phi}^2 |S^h|}\ .\]
Let $a>0$ be a positive constant whose value we determine later. For any $t>0$, with probability larger than $1-4e^{-t}$, we have
\beqn
 \frac{({\rm III})}{\sigma_{\phi}^2} 
 &\geq& \frac{|\bbN_h|- C\left(|\bbN_h|\|\phi\|_1+ \|\bGamma^{-1}(\phi)\|\|\phi\|_2^2+ \frac{|\bbN_h|^{5/2} + |\bbN_h|^2 (\sum_{j\neq 0}\gamma_{j}^2)}{\sqrt{|S^h|}} \right)t^2}{1+
4\|\bGamma(\phi)\|\|\bGamma^{-1}(\phi)\|\frac{|\bbN_h|}{\sqrt{|S^h|}}\left(\sqrt{t}+\log(|\bbN_h|)\right)}\\
&\geq & |\bbN_h|- C\left(|\bbN_h|^{3/2}\|\phi\|_2+ \|\phi\|_2^2+ \frac{|\bbN_h|^{5/2} + |\bbN_h|^2 (\sum_{j\neq 0}\gamma_{j}^2)}{\sqrt{|S^h|}} \right)t^2 - 16\frac{|\bbN_h|}{\sqrt{|S^h|}}\left(\sqrt{t}+\log(|\bbN_h|)\right)\\
&\geq &|\bbN_h|- C\left(|\bbN_h|^{3/2}\|\phi\|_2+ \frac{|\bbN_h|^{5/2} }{\sqrt{|S^h|}} \right)t^2 - C'\frac{|\bbN_h|\log(|\bbN_h|)}{\sqrt{|S^h|}}\left(1+ t^4\right)\\
&\geq & |\bbN_h|- a |S^h|\|\phi\|^2_2 -  C\left(a^{-1}\frac{|\bbN_h|^{3}}{|S^h|}t^4 +\frac{|\bbN_h|^{5/2} }{\sqrt{|S^h|}} t^2 + \frac{|\bbN_h|\log(|\bbN_h|)}{\sqrt{|S^h|}}\left(1+ t^4\right)\right)
\\
 &\geq & |\bbN_h| - a |S^h|\|\phi\|^2_2 - C\sqrt{|\bbN_h|}\left(1+ (a^{-1}+1)t^4\right)\ .
\eeqn
Here in the first line, we used Lemma \ref{lem:control_(III)}.
In the second line, we used the fact that $(1-y)/(1+x)\geq 1-x-y$ for all $x,y \ge 0$, $\|\phi\|_1\leq \sqrt{|\bbN_h|}\|\phi\|_2$ by the Cauchy-Schwarz inequality, and $\|\bGamma(\phi)\|\vee \|\bGamma^{-1}(\phi)\|\leq 2$.
In the third line, we applied the inequality $\sum_{j\neq 0}\gamma_{j}^2\leq 4\|\phi\|_2^4+ 16\|\phi\|_2^2\leq 20$, which is a consequence of $\|\phi\|_1\leq 1/2$ and Lemma \ref{lem:conditional_variance}. The last line is a consequence of   Condition \eqref{eq:condition_LS}. Then, we take $a=C/2$ with $C$ as in \eqref{eq:lower_case1} and apply Lemma \ref{lem:control_(IV)} to control the denominator of $T_S$. This leads to
\[
\P\left\{T_S \geq C |S^h|\|\phi\|_2^2 + |\bbN_h|- C' \sqrt{\bbN_h}(1\vee t^4)\right\}\geq 1 - 4e^{-t} -2e^{-C'' \frac{|S^h|}{|\bbN_h|^2}}\ . \]
Taking $t=\log(8/\beta)$ and letting $C_2$ be small enough in \eqref{alpha_beta}, we get 
\[
\P\left\{T_S \geq C |S^h|\|\phi\|_2^2 + |\bbN_h|- C' \sqrt{\bbN_h}(1+\log^{4}(\beta^{-1}))\right\}\geq 1 - \beta\ , \]
proving \eqref{eq:upper_TS_H1} in Case 1.

%%%%%%%%%%%
\medskip\noindent 
{\bf Case 2}: $\|\phi\|_2^2\geq (4 |\bbN_h|)^{-1}$. This condition entails 
\[\frac{2\|\phi\|_2^2}{1+\|\phi\|_1}\geq  \frac{\|\phi\|_2}{\sqrt{|\bbN_h|}}\ .\]
Since the term $({\rm III})$ is non-negative, we can start from the lower bound $Z_S\geq ({\rm I})+ ({\rm II})$. 
We derive from Lemma \ref{lem:control_(I)} and the above inequality that
\beq\label{eq:lower_I-bis}
\P\left\{({\rm I})\geq |S^h|\frac{\sigma_{\phi}^2\|\phi\|_2}{4\sqrt{|\bbN_h|}} \right\}\geq 1- \exp\left(-C\frac{|S^h|}{|\bbN_h|}\right)\ .
\eeq
Taking $t=C |S^h|/|\bbN_h|^2$ in \eqref{eq:deviation_II-b} for a constant $C$ sufficiently small, and using Condition \eqref{eq:condition_LS}, we get that $({\rm II}) \ge - 3 \sqrt{C} \sigma_\phi \sqrt{|S^h|/|\bbN_h|} \sqrt{({\rm I})}$ with probability at least $1 - e^{-t}$. 
Also, $\|\phi\|_2^2 \ge (4|\bbN_h|)^{-1}$ implies that the right-hand side exceeds $- \frac12 ({\rm I})$ when the event in \eqref{eq:lower_I-bis} holds and $C$ is small enough.  Hence, we get
\[\P\left\{({\rm I}) + ({\rm II})\geq |S^h|\frac{\sigma_{\phi}^2\|\phi\|_2}{8\sqrt{|\bbN_h|}} \right\}\geq 1- 2 \exp\Big(-C\frac{|S^h|}{|\bbN_h|^2}\Big)\ .\]
Finally, we combine this bound with \eqref{eq:upper_epsilon} and the condition $\|\phi\|_2^2 \ge (4|\bbN_h|)^{-1}$, to get
\beqn
\P\left\{T_S \geq \frac{|S^h|}{32|\bbN_h|} \right\}\geq 1- 3 \exp\left(-C\frac{|S^h|}{|\bbN_h|^2}\right) \geq 1-\beta ,
\eeqn
where we used the condition on $\beta$. In view of Condition \eqref{eq:condition_LS}, we have proved \eqref{eq:upper_TS_H1}.
This concludes the proof of \thmref{LS1}.  It remains to prove the auxiliary lemmas.

%\begin{proof}[Proof of Lemma \ref{lem:control_(I)}]
\subsubsection{Proof of Lemma \ref{lem:control_(I)}}
Recall the definition of $S_i$ in \eqref{S_i}.  
Let $\bF_{S_i}$ denote the matrix with row vectors $F_j, j \in S_i$.  We have 
\[({\rm I}) = \|\bF_{S,h}\phi\|^2_2 = \sum_{i\in \bbN_h\cup\{0\}} \|\bF_{S_i}\phi\|^2_2 \ .\]
%Let us lower bound the smallest eigenvalue of the covariance matrix of $\bF_{S_i}\phi$. 
For any $u\in \bbR^{S_i}$, 
\beqn
\Var\left[\sum_{j\in S_i} u_j\bF_{S_i}\phi\right]= \Var\left[\sum_{j\in S_i}\sum_{v\in \bbN_h}u_j\phi_vY_{v+j}\right]\geq \|u\|_2^2\|\phi\|_2^2 \lambda^{\min}(\bGamma(\phi))\ ,
\eeqn
since the indices $(v+j : v\in\bbN_{h}, j\in S_i)$ are all distinct.  
Since $\lambda^{\min}(\bGamma(\phi))\geq \frac{\sigma_{\phi}^2}{1+\|\phi\|_1}$ by Lemma \ref{lem:spectrum_gamma}, 
 $\frac{1+\|\phi\|_1}{\sigma_\phi^2\|\phi\|_2^2}\|\phi^\top \bF_{S_i}\|_2^2$ is stochastically lower bounded by a $\chi^2$ distribution with $|S_i|$ degrees of freedom. By Lemma \ref{lem:normal-quad} and the union bound, we have that for any $t>0$, 
\beqn
({\rm I})&\geq& \frac{\|\phi\|_2^2}{1+\|\phi\|_1}\sigma_{\phi}^2 \sum_{i\in \bbN_h\cup\{0\}} |S_i| -2 \sqrt{|S_i|[\log(|\bbN_h|+1)+t]}\\
&\geq & \frac{\|\phi\|_2^2}{1+\|\phi\|_1}\sigma_{\phi^2 } \left(|S^h|- 2 \sqrt{|S^h|(|\bbN_h|+1)[\log(|\bbN_h|+1)+t]}\right)\ ,
\eeqn
with probability larger than $1-e^{-t}$. Finally we set  $t=\tfrac{|S^h|}{32(|\bbN_h|+1)}$  and use Condition \eqref{eq:condition_LS} to conclude.
%\end{proof}

%\begin{proof}[Proof of Lemma \ref{lem:control_(II)}]
\subsubsection{Proof of Lemma \ref{lem:control_(II)}}
We first prove \eqref{eq:deviation_II-a}. Denote by $\tilde{\bSigma}_{\phi,S}$ the covariance matrix of the random vector  $(\epsilon^\top _{\phi,S},\phi^\top  \bF_{S,h}^\top )$ of size $2|S^h|$. 
Let $\bR$ be the block matrix defined by 
\[\bR=\begin{pmatrix}  0 &\bI_{S^h} \\ \bI_{S^h} & 0    \end{pmatrix}.\] 
Letting $Z$ be a standard Gaussian vector of size $2|S^h|$, we have
$2\phi^\top  \bF_{S,h}^\top \epsilon_{S,h}\sim Z^\top  \tilde{\bSigma}^{1/2}_{\phi,S}\bR  \tilde{\bSigma}^{1/2}_{\phi,S}Z \ .$
From Lemma \ref{lem:normal-quad} we get that for all $t>0$, with probability at least $1-e^{-t}$,
\begin{eqnarray}
2\phi^\top  \bF_{S,h}^\top \epsilon_{S,h}& \le& \tr[\tilde{\bSigma}^{1/2}_{\phi,S}\bR\tilde{\bSigma}^{1/2}_{\phi,S}]- 2\|\tilde{\bSigma}^{1/2}_{\phi,S}\bR\tilde{\bSigma}^{1/2}_{\phi,S}\|_F\sqrt{t}- 2 \|\tilde{\bSigma}^{1/2}_{\phi,S}\bR\tilde{\bSigma}^{1/2}_{\phi,S}\|t\ , \notag \\
%&\geq&  - 2\|\tilde{\bSigma}^{1/2}_{\phi,S}\bR\tilde{\bSigma}^{1/2}_{\phi,S}\|_F\sqrt{t}- 2 \|\tilde{\bSigma}_{\phi,S}\|\|\bR\|t\\
&\le & - 2\|\tilde{\bSigma}^{1/2}_{\phi,S}\bR\tilde{\bSigma}^{1/2}_{\phi,S}\|_F\sqrt{t} - 2 \|\tilde{\bSigma}_{\phi,S}\| t\ , \label{phi_F_eps}
\end{eqnarray}
where we used the fact that $\tr[\tilde{\bSigma}^{1/2}_{\phi,S}\bR\tilde{\bSigma}^{1/2}_{\phi,S}] = \E[\phi^\top  \bF_{S,h}^\top  \epsilon_{S,h}]=0$ and that $\|\bR\| = 1$. In order to bound the Frobenius norm above, we start from the identity
\[
\|\tilde{\bSigma}^{1/2}_{\phi,S}\bR\tilde{\bSigma}^{1/2}_{\phi,S}\|^2_F = \Var[2\phi^\top  \bF_{S,h}^\top  \epsilon_{S,h}] = \E\left[(2 \phi^\top  \bF_{S,h}^\top  \epsilon_{S,h})^2\right] = 4 \sum_{i,j\in S^h}\E[\epsilon_i\epsilon_j (\phi^\top F_i)(\phi^\top F_j)] \ ,
\]
with $\eps_i$ being the $i$th component of $\eps_{S,h}$.
For $i=j$, the expectation of the right-hand side is $\sigma_{\phi}^2\E[(\phi^\top F_i)^2]$, while if the distance between $i$ and $j$ is larger than $h$, then $\epsilon_i$ and $(\epsilon_j,F_i,F_j)$ are independent and the expectation of the right-hand side is zero. If $1 \le |i-j|\leq h$, then we use Isserlis' theorem, together with the fact that $\eps_i \perp F_i$, to obtain
\beqn
|\E[\epsilon_i\epsilon_j (\phi^\top F_i)(\phi^\top F_j)]|&=& \big|\E[\epsilon_i\epsilon_j]\E[(\phi^\top F_i)(\phi^\top F_j)]+ \E[\epsilon_i\phi^\top F_j]E[\epsilon_j\phi^\top F_i]\big|\\
& \leq  & \sigma_{\phi}^2|\phi_{i-j}|\E[(\phi^\top F_i)^2]+ \phi^2_{j-i}\sigma_{\phi}^2\ .
\eeqn
Putting all the terms together, we obtain
\beqn
\|\tilde{\bSigma}^{1/2}_{\phi,S}\bR\tilde{\bSigma}^{1/2}_{\phi,S}\|^2_F
&\leq & 4 |S^h|\sigma_{\phi}^2 \Big\{ \E[(\phi^\top F_i)^2] (1 +\|\phi\|_1) + \|\phi\|_2^2 \Big\} \\
&\leq & 4 \sigma_{\phi}^2|S^h|\|\phi\|_2^2 \|\bGamma(\phi)\|(2+\|\phi\|_1)\ ,
\eeqn
using the fact that $\|\bGamma(\phi)\| \ge 1$.

Turning to $\|\tilde{\bSigma}_{\phi,S}\|$, denote $\bGamma(\phi)^{\epsilon}$ the covariance of the process $(\epsilon_i,\ i\in \bbZ^d)$. By Lemma \ref{lem:covariance_residuals},  
$(\bGamma(\phi)^{\epsilon})_{i,j}=\left[- \phi_{i-j}+ \1_{i=j}\right]\sigma_{\phi}^2$, and  it follows that $\|\bGamma(\phi)^{\epsilon}\|\leq (1+\|\phi\|_1)\sigma_{\phi}^2$.
Then, for all vectors $u,v \in \bbR^{S^h}$,
\beqn
\Var\left(\sum_{i\in S^h} u_i\phi^\top  F_i + \sum_{i\in S^h} v_i\epsilon_i\right)&= & \Var\left(\sum_{i\in S^h} u_iY_i + \sum_{i\in S^h} (v_i-u_i)\epsilon_i\right)\\
&\leq & 2 \Var\left(\sum_{i\in S^h} u_iY_i \right)+ 2\Var\left(\sum_{i\in S^h} (v_i-u_i)\epsilon_i\right)\\
&\leq & 2\|u\|_2^2 \|\bGamma(\phi)\| + 2 \|u-v\|_2^2 \|\bGamma(\phi)^{\epsilon}\|\\
&\leq & 6\left(\|u\|^2_2+ \|v\|_2^2\right) \left[\|\bGamma(\phi)\|\vee \|\bGamma(\phi)^{\epsilon}\|\right]\ .
\eeqn 
Consequently, $\|\tilde{\bSigma}_{\phi,S}\|\leq 6[\|\bGamma(\phi)\|\vee \|\bGamma(\phi)^{\epsilon}\|]\leq 6[\|\bGamma(\phi)\|\vee (1+\|\phi\|_1)\sigma_{\phi}^2]$.

We conclude that \eqref{eq:deviation_II-a} holds by virtue of the two bounds we obtained for the two terms in \eqref{phi_F_eps}.

\medskip
Turning to \eqref{eq:deviation_II-b}, we decompose $({\rm II})$ into  $2 \sum_{i\in \bbN_h\cup\{0\}}\phi^\top  \bF_{S_i}^\top  \epsilon_{S_i}$.
For any $j_1\neq j_2\in S_i$, $|j_1-j_2|_{\infty}\geq 2h+1$ and therefore $\epsilon_{j_1}$ is  independent of $(Y_{j_2+v},\ v\in\bbN_h\cup\{0\})$. Since $\epsilon_{j_2}$ and $F_{j_2} \phi$ are linear combinations of this collection, we conclude that $\epsilon_{j_1}\perp (\epsilon_{j_2}^\top , \phi^\top F_{j_2}^\top )$.
Consequently, $\epsilon_{S_i}/\sigma_\phi$ follows a standard normal distribution and is independent of $\bF_{S_i} \phi$. By conditioning on $\bF_{S_i} \phi$ and applying a standard Gaussian concentration inequality, we get 
\[\P\Big\{|\phi^\top \bF_{S_i}^\top \epsilon_{S_i}| \leq \sigma_\phi\|\bF_{S_i} \phi\|_2\sqrt{2t} \Big\} \leq e^{-t}\ ,\] 
for any $t>0$. We then take a union bound over all $i\in \bbN_h\cup\{0\}$. For any $t>0$, 
\beqn
({\rm II}) &\geq & -2\sqrt{2} \sigma_{\phi}\sqrt{\log(|\bbN_h|+1)+ t} \sum_{i\in \bbN_h\cup\{0\}}\|\phi^\top \bF_{S_i}\|_2\\ 
&\ge & -2\sqrt{2} \sigma_{\phi}\sqrt{\log(|\bbN_h|+1)+ t} \sqrt{|\bbN_h|+1}\|\bF_{S,h}\phi\|_2\ ,
\eeqn
with probability larger than $1-e^{-t}$.
%\end{proof}

%\begin{proof}[Proof of Lemma \ref{lem:control_(III)}]
\subsubsection{Proof of Lemma \ref{lem:control_(III)}}

\noindent{\bf Proof of \eqref{eq:deviation_wisharts}.}
Fix $(v_1,v_2)\in \bbN_h$ and consider the random variable 
\[(\bF_{S,h}^\top  \bF_{S,h})_{v_1,v_2} = \sum_{i\in S^h}Y_{i+v_1}Y_{i+v_2} = Y_{S}^\top \bR Y_S = V^\top \bGamma_S(\phi)^{1/2}\bR \bGamma^{1/2}_{S}(\phi)V\ ,\] 
which constitutes a definition for the symmetric matrix $\bR$, and $V\sim \cN(0,\bI_S)$. Observe that $\|\bR\|_F^2=|S^h|$ and $\|\bR\|\leq 1$ as the $l_1$ norm of each row of $\bR$ is smaller than one.
We derive from Lemma \ref{lem:normal-quad}, and the fact that $\|\bGamma_S(\phi)^{1/2}\bR \bGamma^{1/2}_{S}(\phi)\|_F^2\leq \|\bR\|_F^2\|\bGamma^{1/2}_{S}(\phi)\|^4\leq |S^h|\|\bGamma(\phi)\|^2$ and $\|\bGamma_S(\phi)^{1/2}\bR \bGamma^{1/2}_{S}(\phi)\|\leq \|\bR\|\|\bGamma_S(\phi)\|\leq \|\bGamma(\phi)\|$, that for any $t>0$,
\[\P\left\{\big|(\bF_{S,h}^\top  \bF_{S,h})_{v_1,v_2}- |S^h|\gamma_{v_1,v_2}\big|\leq 2\|\bGamma(\phi)\|\sqrt{|S^h|t}+ 2\|\bGamma(\phi)\|t \right\}\leq 2e^{-t}~.\]
Then we bound the $\ell_2$ operator norm of $|S^h|^{-1}\bF_{S,h}^\top  \bF_{S,h}- \bGamma_{\bbN_h}(\phi)$ by its $\ell_1$ operator norm and combine the above deviation inequality with a union bound over all $(v_1,v_2)\in \bbN_h$. 
Thus, for any $t\leq |S^h|$, 
\beqn
\left\|\frac{\bF_{S,h}^\top  \bF_{S,h}}{|S^h|}- \bGamma_{\bbN_h}(\phi)\right\|&\leq& \sup_{v_1\in \bbN_h}\sum_{v_2\in \bbN_h}\Bigg|\frac{(\bF_{S,h}^\top  \bF_{S,h})_{v_1,v_2}}{|S^h|}- \gamma_{v_1,v_2}\Bigg|\\
&\leq & 2\|\bGamma(\phi)\|\frac{|\bbN_h|}{|S^h|^{1/2}}\left( \sqrt{\log |\bbN_h| + t}+ \frac{\log |\bbN_h| + t}{\sqrt{|S^h|}}\right) \\
&\leq & 4\|\bGamma(\phi)\|\frac{|\bbN_h|}{|S^h|^{1/2}}\left(\sqrt{t}+ \log(|\bbN_h|)\right) \ ,
\eeqn
with probability larger than $1-2e^{-t}$. Hence, under this event,
\[\lambda^{\rm max}\left(\bGamma_{\bbN_h}(\phi)^{-1/2}\frac{\bF_{S,h}^\top  \bF_{S,h}}{|S^h|}\bGamma_{\bbN_h}(\phi)^{-1/2}\right)\leq 1+  4\|\bGamma(\phi)\|\|\bGamma^{-1}(\phi)\|\frac{|\bbN_h|}{|S^h|^{1/2}}\left(\sqrt{t}+\log(|\bbN_h|)\right)\ ,\]
since $\|\bGamma_{\bbN_h}(\phi)^{-1}\|\leq \|\bGamma^{-1}(\phi)\|$.
This concludes the proof of \eqref{eq:deviation_wisharts}.

\medskip\noindent{\bf Proof of \eqref{eq:control_bSigma}.} 
Turning to the second deviation bound, we use the following decomposition 
\[
\|\bGamma_{\bbN_h}(\phi)^{-1/2}\bF_{S,h}^\top  \epsilon_{S,h}\|_2^2= \sum_{i\in S^h} \epsilon_{i}^2\|\bGamma_{\bbN_h}(\phi)^{-1/2}F_{i}\|_2^2+ \sum_{(i,j),\ i\neq j } \epsilon_i\epsilon_j F_j^\top  \bGamma_{\bbN_h}(\phi)^{-1}F_i =: A+ B\ ,
\]
with $\eps_i$ being the $i$th entry of $\eps_{S,h}$.
Since both $A$ and $B$ are Gaussian chaos variables of order 4, we apply Lemma \ref{lem:chaos_4} to control their deviations. For any $t>0$, 
\beq\label{eq:decomposition_A+B}
\P\left\{A+B\geq \E[A+B] - C\left(\Var^{1/2}(A)+ \Var^{1/2}(B)\right)t^2 \right\}\leq 2e^{-t}\ ,
\eeq
using the fact that $\Var^{1/2}(A+B) \le \Var^{1/2}(A)+ \Var^{1/2}(B)$.
Thus, it suffices to compute the expectation and variance of $A$ and $B$. 

%%%%%%%
First, we have $\E[A]= |S^h||\bbN_h|\sigma_{\phi}^2$, by independence of $\eps_i$ and $F_i$, and from this we get
\beqn
\Var(A)&=&\sum_{i,j \in S^h}\Big(\E\left[\epsilon_{i}^2\epsilon_{j}^2 \|\bGamma_{\bbN_h}(\phi)^{-1/2}F_{i}\|_2^2\|\bGamma_{\bbN_h}(\phi)^{-1/2}F_{j}\|_2^2\right]- \sigma^4_{\phi}|\bbN_h|^2 \Big) =: \sum_{i,j \in S^h}A_{i,j}\ . 
\eeqn
If  $|i-j|_{\infty}\leq h$, we may use the Cauchy-Schwarz inequality to get
\beqn
|A_{i,j}|\leq \E\left[\epsilon_{i}^4  \|\bGamma_{\bbN_h}(\phi)^{-1/2}F_{i}\|_2^4\right]= 3\sigma_{\phi}^4|\bbN_h|(|\bbN_{h}|+2)\ ,
\eeqn
again by independence of $\epsilon_i$ and $\|\bGamma_{\bbN_h}(\phi)^{-1/2}F_{i}\|_2^2$. 
If $|i-j|_{\infty}> h$, then $\epsilon_i$ is independent of $(F_i,F_j,\epsilon_j)$ and $\epsilon_j$ is independent of $(F_i,F_j,\epsilon_i)$, so we get
\beqn
\frac{A_{i,j}}{\sigma_\phi^4}
&=& \E \left[\|\bGamma_{\bbN_h}(\phi)^{-1/2}F_{i}\|_2^2\|\bGamma_{\bbN_h}(\phi)^{-1/2}F_{j}\|_2^2\right]- |\bbN_h|^2\\
& = & \hspace{-0.2in} \sum_{v_1,v_2,v_3,v_4\in \bbN_h} \hspace{-0.2in} (\bGamma_{\bbN_h}(\phi)^{-1})_{v_1,v_2}(\bGamma_{\bbN_h}(\phi)^{-1})_{v_3,v_4}\left[\gamma_{v_1-v_2}\gamma_{v3-v_4}+ \gamma_{i+v_1-j-v_3}\gamma_{i+v2-j-v_4}+ \gamma_{i+v_1-j-v_4}\gamma_{i+v3-j-v_2}\right]\\ & &- |\bbN_h|^2\ \\
& =&  \hspace{-0.2in} \sum_{v_1,v_2,v_3,v_4\in \bbN_h} (\bGamma_{\bbN_h}(\phi)^{-1})_{v_1,v_2}(\bGamma_{\bbN_h}(\phi)^{-1})_{v_3,v_4}\left[ \gamma_{i+v_1-j-v_2}\gamma_{i+v3-j-v_4}+ \gamma_{i+v_1-j-v_4}\gamma_{i+v3-j-v_2}\right]
\eeqn
where we apply Isserlis' theorem in the second  line and use the definition of  $\bGamma_{\bbN_h}(\phi)$ in the last line. 
By symmetry, we get
\beqn
\frac{|A_{i,j}|}{\sigma_\phi^4}&\leq& 2\|\bGamma_{\bbN_h}(\phi)^{-1}\|_{\infty}^2\sum_{v_1,v_2,v_3,v_4\in \bbN_h}|\gamma_{i+v_1-j-v_3}\gamma_{i+v2-j-v_4}|\\
&\leq & 2\|\bGamma_{\bbN_h}(\phi)^{-1}\|^2|\bbN_h| \sum_{v_1,v_2\in \bbN_h}\gamma^2_{i-j+v_1-v_2}\\
&\leq & 2\|\bGamma^{-1}(\phi)\|^2 |\bbN_h|^2 \sum_{v\in \bbN_{2h}}\gamma^2_{i-j+v}\ ,
\eeqn
 using the Cauchy-Schwarz inequality in the second line. Here $\|\bA\|_{\infty}$ denotes the supremum norm of the entries of $\bA$.
Then, summing over all $j$ lying at a distance larger than $h$ from $i$,
\beqn
\sum_{j\in S^h, \, |j-i|_{\infty}>h}\frac{|A_{i,j}|}{\sigma_\phi^4} &\leq&  2\|\bGamma^{-1}(\phi)\|^2 |\bbN_h|^2  \sum_{j\in S^h,\,  |j-i|_{\infty}>h}\ \sum_{v\in \bbN_{2h}}\gamma^2_{i-j+v}\\
&\leq &  2^{d+1}\|\bGamma^{-1}(\phi)\|^2 |\bbN_h|^3 \sum_{j\in\mathbb{Z}^d\setminus\{0\}}\gamma_j^{2}\ .
\eeqn
Putting the terms together, we conclude that
\beq\label{eq:var_A}
\Var(A)\leq \sigma_{\phi}^4|S^h||\bbN_h|^3\left(6+ 2^{d+1}\sum_{j\in\mathbb{Z}^d\setminus\{0\}}\gamma_j^{2}\right)\ .
\eeq

%%%%%%%
Next we bound the first two moments of $B$. 
Consider $(i,j)\in S^h$ such that  $|i-j|_{\infty}>h$.  Then $\E\big[\epsilon_i\epsilon_j F_j^\top  \bGamma_{\bbN_h}(\phi)^{-1}F_i\big] = 0$ by independence of $\epsilon_i$ with the other variables in the expectation. 
Suppose now that $|i-j|_{\infty}\leq h$. By Isserlis' theorem, and the independence of $\eps_i$ and $F_i$, as well as $\eps_j$ and $F_j$, and symmetry, to get
\beqn
\E\left[\epsilon_i\epsilon_j F_j^\top  \bGamma_{\bbN_h}(\phi)^{-1}F_i\right]&=& \E\left[\epsilon_iF_j^{T}\right]\bGamma_{\bbN_h}(\phi)^{-1}\E\left[F_i\epsilon_j\right]+ \E\left[\epsilon_i\epsilon_j\right]\E\left[F_j^\top  \bGamma_{\bbN_h}(\phi)^{-1}F_i\right]\\
&\geq & - \sigma^4_{\phi}|\phi_{i-j}|^2\|\bGamma_{\bbN_h}(\phi)^{-1}\| -\sigma_{\phi}^2|\phi_{i-j}| |\bbN_h|\ ,
\eeqn
using the Cauchy-Schwarz inequality and Lemma \ref{lem:covariance_residuals}.
As a consequence, 
\beq\label{eq:E_B}
\E[B] \geq  -\sigma_{\phi}^4|S^h|\|\phi\|_2^2\|\bGamma^{-1}(\phi)\| - \sigma_{\phi}^2|S^h||\bbN_h|\|\phi\|_1\ .
\eeq
Turning to the variance, we obtain
\[\Var(B)\leq \E[B^2] = \sum_{i_1\neq i_2} \sum_{i_3\neq i_4 }\E[V_{i_1,i_2,i_3,i_4}]\ ,\]
where
\[V_{i_1,i_2,i_3,i_4} := \epsilon_{i_1}\epsilon_{i_2}\epsilon_{i_3}\epsilon_{i_4} F_{i_1}^\top  \bGamma_{\bbN_h}(\phi)^{-1}F_{i_2} F_{i_3}^\top  \bGamma_{\bbN_h}(\phi)^{-1}F_{i_4}\ .\]
Fix $i_1$. If one index among $(i_1,i_2,i_3,i_4)$ lies at a distance larger than $h$ from the three others, then the expectation of $V_{i_1,i_2,i_3,i_4}$ is equal to zero. If one index lies within distance $h$ of $i_1$ and the two remaining indices lie within distance $3h$ of $i_1$, we use the Cauchy-Schwarz inequality to get
\[\E[V_{i_1,i_2,i_3,i_4}]
\le \E\Big[\prod_{k=1}^4 |\epsilon_{i_k}| (F_{i_k}^\top \bGamma_{\bbN_h}(\phi)^{-1}F_{i_k})^{1/2}\Big] \\
\leq \E[\epsilon_1^4(F_1^\top \bGamma_{\bbN_h}(\phi)^{-1}F_1)^2]= 3\sigma_{\phi}^4|\bbN_h|(|\bbN_h|+2)\ .
\]
Finally, if say $|i_1-i_2|_{\infty}\leq h$ and $|i_3-i_4|_{\infty}\leq h$ and $|i_k - i_\ell| > h$ for $k = 1,2$ and $\ell = 3,4$,
then we use again Isserlis' theorem and simplify the terms to get
\beqn
\E[V_{i_1,i_2,i_3,i_4}]
&=& \E[\epsilon_{i_1}\epsilon_{i_2}]\E[\epsilon_{i_3}\epsilon_{i_4}]\E[F_{i_1}^\top \bGamma_{\bbN_h}(\phi)^{-1}F_{i_2}F_{i_3}^\top \bGamma_{\bbN_h}(\phi)^{-1}F_{i_4}]\\ 
&+& \E[\epsilon_{i_2}F^\top _{i_1}]\bGamma_{\bbN_h}(\phi)^{-1}\E[F_{i_2}\epsilon_{i_1}] \E[\epsilon_{i_4}F^\top _{i_3}]\bGamma_{\bbN_h}(\phi)^{-1}\E[F_{i_4}\epsilon_{i_3}]\\
&\leq & \sigma_{\phi}^4|\phi_{i_2-i_1}\phi_{i_4-i_3}||\bbN_h|(|\bbN_h|+2)+ \|\bGamma^{-1}(\phi)\|^2\sigma_{\phi}^8 \phi_{i_2-i_1}^2\phi_{i_4-i_3}^2 \ ,
\eeqn
where we used again Lemma \ref{lem:covariance_residuals} to control the terms involving $\epsilon$'s and the Cauchy-Schwarz inequality to bound the term in $(F_{i_k}, k=1,\ldots,4)$.
Putting all the terms together, we conclude that
\beq\label{eq:var_B}
\Var(B)\leq C \sigma_{\phi}^4\left(|S^h||\bbN_h|^5 + |S^h|^2\|\phi\|_1^2|\bbN_h|^2+ |S^h|^2\sigma_{\phi}^2\|\phi\|_2^4 \|\bGamma^{-1}(\phi)\|^2\right)\ ,
\eeq
since $\sigma_{\phi}^2\leq \Var[Y_i]= 1$. 

Plugging in the bounds that we obtained for the moments of $A$ and $B$ in \eqref{eq:decomposition_A+B}, we conclude the proof of \eqref{eq:control_bSigma}.
%\end{proof}

%\begin{proof}[Proof of Lemma \ref{lem:control_(IV)}]
\subsubsection{Proof of Lemma \ref{lem:control_(IV)}}
Recall the definition of $S_i$ in \eqref{S_i}.  
We decompose $\|\epsilon_{S,h}\|_2^2= \sum_{i\in \bbN_h\cup\{0\}}\|\epsilon_{S_i}\|^2_2$ and note that $\|\epsilon_{S_i}\|^2_2 \sim \sigma_\phi^2 \chi^2_{|S_i|}$. 
Applying the second deviation bound of Lemma~\ref{lem:normal-quad} together with a union bound, we obtain that for any $t>0$, 
\beqn
\|\epsilon_{S,h}\|_2^2 &\leq& \sigma_{\phi}^2 \sum_{i\in \bbN_h\cup \{0\}} \Big(|S_i|+ 2\sqrt{|S_i|\log(|\bbN_h|+1)+ t}+ 2 t+ 2\log(|\bbN_h|+1) \Big)\\
&\leq & \sigma_{\phi}^2\left(|S^h| + 2\sqrt{|S^h|(|\bbN_h|+1) \log(|\bbN_h|+1)+ t) }+ 2|\bbN_h|\left(t+ \log(|\bbN_h|+1)\right)\right)\ ,
\eeqn
with probability larger $1-e^{-t}$. Relying on Condition \eqref{eq:condition_LS}, we derive that 
\beqn
\P\left\{\|\epsilon_{S,h}^2\|\leq \sigma_{\phi}^2|S^h|(1+ |\bbN_h|^{-1/2})\right\} \geq 1-\exp\left(-C \frac{|S^h|}{|\bbN_h|^2}\right)\ ,
\eeqn
for a numerical constant $C>0$ small enough.
%\end{proof}

\subsection{Proof of Corollary \ref{cor:AR}}
It is well known---see, e.g., \cite{MR1419991}---that any $\ar_h$ process is also a Gaussian Markov random field with neighborhood radius $h$ (and vice-versa). 
Denote $\tau_\psi^2$ the innovation variance  of an $\ar_h(\psi)$ process. 
The bijection between the parameterizations $(\psi,\tau_\psi^2)$ and $(\phi,\sigma_{\phi}^2)$ is given by the following equations
\begin{eqnarray}
\phi_{-i}~=~\phi_i&=&\frac{\psi_{i}- \sum_{k=i+1}^h \psi_k \psi_{k-i}}{1+ \|\psi\|_2^2}~,\quad \quad \text{for $i=1,\ldots, h$ ,} \label{eq:phi_psi_ar} \\
\sigma^2_{\phi}&=&  \frac{\tau_\psi^2}{1+\|\psi\|_2^2} \ .\label{eq:variance_sigma_ar}
\end{eqnarray}
This correspondence is maintained below.

\noindent 
{\bf Lower bound}. In this proof, $C$ is a positive constant that may vary from line to line. 
It follows from the above  equations that 
\[\|\phi\|_2^2 \leq C\frac{\|\psi\|_2^2 + h\|\psi\|_2^4}{1+\|\psi\|_2^2} \ .\]
Consider any $r\leq 1/h$. 
In that case, if $\|\phi\|_2\geq r$ then the inequality above implies that $\|\psi\|_2\geq Cr$, and as a consequence, $R^*_{\cC,\mathfrak{G}(h,r)}\leq R^*_{\cC,\mathfrak{F}(h,Cr)}$.  Therefore, since \eqref{AR1} and our condition on $h$ together imply that $r \le 1/h$ eventually, it suffices to prove that $R^*_{\cC,\mathfrak{G}(h,r)} \to 1$.
For that, we apply Corollary~ \ref{cor:lower_hypercube}.  Condition \eqref{eq:condition_neigbhorhood} there is satisfied eventually under our assumptions (\eqref{AR1} and our condition on $h$).
Consequently, we have $R^*_{\cC,\mathfrak{G}(h,r)}\to 1$ as soon as \eqref{eq:powerless_gmrf} holds, which is the case when \eqref{AR1} holds.

\medskip 
\noindent 
{\bf Upper bound}. It follows from \eqref{eq:variance_sigma_ar} and the inequality $\tau_\psi^2\leq 1$ that 
\[1- \sigma^2_{\phi}\geq \frac{\|\psi\|_2^2}{1+\|\psi\|_2^2}\ . 
\]
Denoting $u_n:=\log(n)/k + \sqrt{h\log(n)}/k$, observe as above that $u_n\ll 1/h$ by our assumption on $h$. 

Assume that $\|\psi\|_2^2\geq r^2$ for some  $r^2\geq u_n$. 
If $\|\phi\|_1\leq 1/2$, it follows from the inequality $1-\sigma^2_{\phi}\leq \|\phi\|_2^2/(1-\|\phi\|_1)\leq 2\|\phi\|^2_2$ (Lemma \ref{lem:spectrum_gamma}) that $\|\phi\|_2^2\geq r^2/4$. 
And if $\|\phi\|_1> 1/2$, then $\|\phi\|_2^2 \geq (8h)^{-1}$ by the Cauchy-Schwarz inequality.
Thus, when $r^2 \le 1/h$, we have $\|\phi\|_2^2\geq r^2/8$, and this implies
\[R_{\cC,\mathfrak{F}(h,r)}(f)\leq R_{\cC,\mathfrak{G}(h,r/\sqrt{8})}(f)~, \quad \text{for any test $f$.}\] 
When $r^2\geq 1/h$, we simply use a monotonicity argument 
\[R_{\cC,\mathfrak{F}(h,r)}(f)\leq R_{\cC,\mathfrak{F}(h,h^{-1/2})}(f)\leq R_{\cC,\mathfrak{G}(h,1/\sqrt{8h})}(f)~, \quad \text{for any test $f$.}\] 
The result then follows from  Theorem \ref{thm:LS1}. 

\subsection*{Acknowledgements}

This work was partially supported by the US National Science Foundation (DMS-1223137, DMS-1120888) and the French Agence Nationale de la Recherche (ANR 2011 BS01 010 01 projet Calibration).
The third author was supported by the Spanish Ministry of Science and Technology grant MTM2012-37195.

\bibliographystyle{chicago}
\bibliography{ref}

\end{document}